\documentclass[11pt, a4paper]{article}
\usepackage{microtype}
\usepackage{amsmath}
\usepackage{amssymb, units, amsthm}
\usepackage[english]{babel}
\usepackage[cp1250]{inputenc}
\usepackage{color}
\usepackage{nicefrac}
\usepackage[numbers,sort&compress]{natbib}
\usepackage{csquotes}

\allowdisplaybreaks

\MakeOuterQuote{"}

\usepackage{gensymb}
\usepackage{subfig}
\usepackage{tikz}
\usetikzlibrary{decorations.pathmorphing,shapes}

\newcounter{sarrow}
\newcommand\xrsquigarrow[1]{%
\stepcounter{sarrow}%
\begin{tikzpicture}[decoration=snake]
\node (\thesarrow) {\strut#1};
\draw[->,decorate] (\thesarrow.south west) -- (\thesarrow.south east);
\end{tikzpicture}%
}

\newtheoremstyle{vety}	
  {5mm}			
  {5mm}			
  {\itshape}		
  {}			
  {}			
  {\bf}		
  {.5em}		
  {\thmname{\textbf{#1}}\thmnumber{\textbf{ #2}}\thmnote{\textnormal{ (#3)}}}

\theoremstyle{vety}
\newtheorem{thm}{Theorem}
\newtheorem{lem}[thm]{Lemma}
\newtheorem{ass}[thm]{Assumption}

\newtheoremstyle{definice}	
  {5mm}				
  {5mm}				
  {\normalfont}			
  {}				
  {}				
  {\bf}			
  {.5em}			
  {\thmname{\textbf{#1}}\thmnumber{\textbf{ #2}}\thmnote{\textnormal{ (#3)}}}

  \newtheorem{innercustomthm}{Lemma}
\newenvironment{customthm}[1]
  {\renewcommand\theinnercustomthm{#1}\innercustomthm}
  {\endinnercustomthm}

\makeatletter
\newcommand\ackname{Acknowledgements}
\if@titlepage
  \newenvironment{acknowledgements}{%
      \titlepage
      \null\vfil
      \@beginparpenalty\@lowpenalty
      \begin{center}%
        \bfseries \ackname
        \@endparpenalty\@M
      \end{center}}%
     {\par\vfil\null\endtitlepage}
\else
  \newenvironment{acknowledgements}{%
      \if@twocolumn
        \section*{\abstractname}%
      \else
        \small
        \begin{center}%
          {\bfseries \ackname\vspace{-.5em}\vspace{\z@}}%
        \end{center}%
        \quotation
      \fi}
      {\if@twocolumn\else\endquotation\fi}
\fi
\makeatother

\makeatletter
\newcommand\poename{Poetry}
\if@titlepage
  \newenvironment{poetry}{%
      \titlepage
      \null\vfil
      \@beginparpenalty\@lowpenalty
      \begin{center}%
        \bfseries \poename
        \@endparpenalty\@M
      \end{center}}%
     {\par\vfil\null\endtitlepage}
\else
  
\fi
\makeatother

\makeatletter
\newcommand\keyname{Keywords}
\if@titlepage
  \newenvironment{keywords}{%
      \titlepage
      \null\vfil
      \@beginparpenalty\@lowpenalty
      \begin{center}%
        \bfseries \keyname
        \@endparpenalty\@M
      \end{center}}%
     {\par\vfil\null\endtitlepage}
\else
  \newenvironment{keywords}{%
      \if@twocolumn
        \section*{\abstractname}%
      \else
        \small
        \begin{center}%
          {\bfseries \keyname\vspace{-.5em}\vspace{\z@}}%
        \end{center}%
        \quotation
      \fi}
      {\if@twocolumn\else\endquotation\fi}
\fi
\makeatother

\theoremstyle{definice}

\newtheorem{rem}[thm]{Remark}

\DeclareMathOperator{\dir}{div}

\DeclareMathOperator{\tr}{Tr}

\DeclareMathOperator{\sym}{sym}

\DeclareMathOperator{\llimsup}{\mathbf{lim\,sup}}

\newcommand{\dfe}{\mathrel=\!{\mathop:} \ }
\newcommand{\ddf}{\mathrel{\mathop:}=}
\DeclareMathOperator*{\esssup}{ess\ sup}

\DeclareMathOperator*{\esslim}{ess\ lim}
\DeclareMathOperator*{\essliminf}{ess\ lim\,inf}

\def\mfrac#1#2{\frac{#1}{#2})}

\def\ocirc#1{\ifmmode\setbox0=\hbox{$#1$}\dimen0=\ht0 \advance\dimen0
  by1pt\rlap{\hbox to\wd0{\hss\raise\dimen0
  \hbox{\hskip.2em$\scriptscriptstyle\circ$}\hss}}#1\else {\accent"17 #1}\fi}

\def \0{\boldsymbol{0}}
\def \A{\mathcal{A}}
\def \Aa{\boldsymbol{A}}
\def \b{\boldsymbol{b}}
\def \BB{\boldsymbol{B}}

\def \B{\mathcal{B}}
\def \bs{\overline{\boldsymbol{S}}}

\def \C{\mathcal{C}}
\def \CC{\boldsymbol{C}}
\def \d{\delta}
\def \D{\boldsymbol{D}}

\def \dk{\overline{\boldsymbol{D}}^k}

\def \dt{\partial_t}

\def \e{\varepsilon}
\def \ek{e^k}
\def \Ek{E^k}
\def \f{\boldsymbol{f}}

\def \ge{\gamma(e)}

\def \ig{\int_{\Gamma}}
\def \igt{\int_{\Gamma_t}}
\def \ii{\int_{\Omega}}
\def \iii{\int_{\partial \Omega}}
\def \intej{\int_{E^j}}
\def \iq{\int_{Q}}
\def \iqt{\int_{Q_t}}
\def \itt{\int_0^T}
\def \I{\boldsymbol{I}}
\def \ke{\kappa(e)}
\def \kk{\kappa(\ek)}

\def \Li{L^{\infty}}
\def \Ln{L^{2}_{\n,\dir}}

\def \M{\mathcal{M}}
\def \n{\boldsymbol{n}}
\def \N{\mathbb{N}}
\def \ne{\nu(e)}

\def \nn{\nabla}
\def \om{\Omega}
\def \pk{p^k}

\def \Pk{\mathcal{P}^k}
\def \q{\boldsymbol{q}}
\def \R{\mathbb{R}}
\def \Rt{\R^{3\times 3}_{\sym}}
\def \s{\boldsymbol{s}}
\def \S{\boldsymbol{S}}

\def \sged{\sigma_2(e)}
\def \sgej{\sigma_1(e)}

\def \sk{\boldsymbol{s}^k}

\def \Sk{\boldsymbol{S}^k}

\def \t{\boldsymbol{T}}

\def \tej{\tau_1(e)}
\def \ted{\tau_2(e)}

\def \ve{\boldsymbol{v}}

\def \vok{(\vk \otimes \vk)\Phi_k(|\vk|)}
\def \vokk{(\ve \otimes \ve)\Phi_k(|\ve|)}
\def \vokn{(\ve^{k_n} \otimes \ve^{k_n})\Phi_{k_n}(|\ve^{k_n}|)}
\def \vk{\boldsymbol{v}^k}
\def \vov{\ve \otimes\ve}
\def \vt{\ve_{\tau}}

\def \w{\boldsymbol{w}}
\def \wk{\boldsymbol{w}^k}

\def \Wn{W^{1,2}_{\n}}
\def \Wnd{W^{1,2}_{\n,\dir}}
\def \Wnn{W_{\n}}

\def \z{\boldsymbol{z}}

\begin{document}
\title{On a Navier-Stokes-Fourier-like system capturing transitions between viscous and inviscid fluid regimes and between no-slip and perfect-slip boundary conditions}

\author{
Erika Maringov\'a\footnote{Charles University, Faculty of Mathematics and Physics, Mathematical Institute, Sokolovsk\'a 83, 186 75 Prague 8, Czech Republic, \tt{maringova@karlin.mff.cuni.cz}},
Josef \v Zabensk\'y\footnote{Corresponding author; University of W\"urzburg, Institute of Mathematics, Chair of Mathematics XI, Emil-Fischer-Stra\ss e 40, 97074 W\"urzburg, Germany, \tt{josef.zabensky@gmail.com}}
}

\date{\today}

\maketitle

\begin{abstract}
\noindent We study a generalization of the Navier-Stokes-Fourier system for an incompressible fluid where the deviatoric part of the Cauchy stress tensor is related to the symmetric part of the velocity gradient via a maximal monotone 2-graph that is continuously parametrized by the temperature. As such, the considered fluid may go through transitions between three of the following regimes: it can flow as a Bingham fluid for a specific value of the temperature, while it can behave as the Navier-Stokes fluid for another value of the temperature or, for yet another temperature, it can respond as the Euler fluid until a certain activation initiates the response of the Navier-Stokes fluid. At the same time, we regard a generalized threshold slip on the boundary that also may go through various regimes continuously with the temperature. All material coefficients like the dynamic viscosity, friction or activation coefficients are assumed to be temperature-dependent. We establish the large-data and long-time existence of weak solutions, applying the $L^{\infty}$-truncation technique to approximate the velocity field.  
\end{abstract}


\section{Introduction}
We study internal unsteady flows of an incompressible heat-conducting fluid 
in three-dimensional bounded domains. The balance equations governing such flows are completed by implicit constitutive relations that characterize rheological properties of the fluid and boundary conditions. The main result of our work is the large-data and long-time existence analysis for the considered class of problems.

\paragraph{Formulation of the problem.} Let $\Omega \subset \mathbb{R}^3$ be a bounded domain and $[0,T)$ be the time interval for a given $T>0$. Denote $Q \ddf [0,T)\times \Omega$ and $\Gamma \ddf [0,T)\times \partial \Omega$ for brevity. Balance equations for flows of an incompressible heat-conducting fluid, 
representing the balance of its linear momentum, the constraint of incompressibility and the balance of energy read
\begin{align}
 \dt \ve + \dir( \vov) - \dir \S + \nn p &= \b \phantom{\b \cdot \ve} \qquad \text{in }Q, \label{system1} \\
\dir \ve & = 0 \phantom{\b \cdot \ve} \qquad \text{in } Q, \label{system2} \\
\dt E + \dir ((E+p)\ve -\S\ve - \ke \nn e ) &= \b \cdot \ve \phantom{\b} \qquad \text{in }Q.  \label{system3}
\end{align}
Here, $\ve:Q\to\R^3$ represents the velocity field, $p:Q\to\R$ is the pressure, $\S:Q\to\Rt$ the deviatoric part of the Cauchy stress\footnote{That is, denoting $\t$ the Cauchy stress, $\S = \t - \nicefrac{1}{3}(\tr \t)\I = \t - p\I$ with $\I$ being the identity tensor.}, $\b:Q\to\R^3$ the external body forces, $\D\ve$ the symmetric part of the velocity gradient, i.e.\ $\D\ve= \nicefrac{1}{2} (\nabla \ve + \nabla^T \ve)$, and $E$ is the sum of the kinetic and the internal energy $e$, that is 
$$E=\frac{|\ve|^2}{2}+e. $$ 
We assume throughout the whole paper that $e$ is proportional to the temperature and work with $e$ rather than with the temperature. Hence, the coefficient $\ke$ stands for the heat conductivity\footnote{Customarily, heat flux $\tilde{\q}(\theta)$ is described by the Fourier law $\tilde{\q}(\theta)=\tilde{\kappa}(\theta)\nabla \theta$ and then $\tilde{\kappa}$ is called \emph{heat conductivity}. Since we consider $\theta\mapsto e(\theta)$ to be a diffeomorphism (in general we have $e(\theta)= c(\theta) \theta$, where $c(\theta)$ is known as \emph{heat capacity}), we can find $\kappa(\cdot)$ such that $\tilde{\kappa}(\theta)\nabla \theta =\kappa(e(\theta))\nabla e(\theta)$. Therefore, we still call $\kappa(e)$ heat conductivity in this paper. For the same reason, in our contemplation we often refer to various temperature-dependent coefficients albeit we always handle functions of the internal energy only. Our reason is an arguably better intuition behind  a temperature-dependent entity.}. The balance of internal energy takes form
\begin{align} \label{boie}
\dt e + \dir (e \ve) - \dir( \ke \nn e) = \S \cdot \D\ve \qquad \text{in}~Q.
\end{align}
There are several forms of the balance of energy. They are equivalent for regular enough functions, but may differ within the context of merely integrable functions: some of them hold for larger class of functions than others. This feature can be successfully exploited; see Feireisl~\cite{feireisl2004dynamics}, Feireisl and M\'alek~\cite{feireisl2006navier} or Bul\'{i}\v cek et al.~\cite{bulivcek2009navier} for more details. 

Unfortunately, we cannot find a solution satisfying~\eqref{boie} for reasons of lacking compactness of its right-hand side in $L^1$ and therefore we will deal with a more amiable balance of total energy~\eqref{system3} (see~\cite{bulivcek2009navier} for a corresponding discussion). So long as we have a smooth solution and the equations~\eqref{system1} and~\eqref{system2} hold, the balance of energies~\eqref{system3} and~\eqref{boie} are actually equivalent. Inasmuch as we deal merely with weak solutions here, this equivalence cannot in general be verified. However, we still require~\eqref{boie} to be satisfied at least as an inequality (in the weak sense; see~\eqref{ttreti}) 
\begin{align} \label{boie2}
\dt e + \dir (e \ve) - \dir( \ke \nn e) \geq \S \cdot \D\ve \qquad \text{in}~Q.
\end{align}
We assume that the boundary is impermeable and require that there is no heat flux across the boundary, i.e.
\begin{align}
\ve \cdot \n  &= 0\phantom{\ve_0 e_0} \quad \text{on}~ \Gamma, \label{system7} \\
\ke \nn e \cdot \n &= 0\phantom{\ve_0 e_0} \quad \text{on}~ \Gamma, \label{system8} \\
\intertext{where $\n$ stands for the outward unit normal vector to $\om$. For given $\ve_0: \om \to \R^3$ and $e_0: \om \to \R$, the initial conditions are}
\ve(0) &= \ve_0\phantom{0e_0} \quad \text{in}~\om, \label{system9} \\
e(0) &= e_0\phantom{0\ve_0} \quad \text{in}~\om.\label{system10}
\end{align}
To close the system, we have yet to provide a constitutive relation involving $\S$ and $-(\S\n)_{\tau}$ along the boundary, denoting 
$$\z_{\tau}\ddf \z - (\z\cdot \n)\n \quad \text{for $\z:\partial \om \to \R^3$},$$ 
which is a key component of our work. We consider these relations to be not expressible as single-valued functions. Starting with the Cauchy stress, an immense advantage of the resulting implicit character lies in enabling us to include quite a general kind of rheology for the investigated fluid. Our model is able to capture the response of a Bingham fluid, Newtonian response and activated Euler fluids. Most importantly, it is capable of a continuous transition between these regimes solely by changing the temperature: the quantities $\S$, $\D\ve$ and $e$ are related to each other by means of a maximal monotone 2-graph parametrised by $e$ (see e.g.~\cite{BGMG,SS14}). More precisely, the triplet $(\S,\D\ve,e)$ is connected through the condition $(\S,\D\ve,e)\in\A \subset \mathbb{R}^{3 \times 3} \times \mathbb{R}^{3 \times 3} \times \R$ a.e.\ in $Q$, where $\A$ is defined through
\begin{align} \label{satis1}
 (\S,\D\ve,e)\in\A \quad &\Longleftrightarrow \quad 2\ne \frac{\left(|\D\ve|-\tau_1(e) \right)^+}{|\D\ve|}\D\ve= \frac{\left(|\S|-\tau_2(e) \right)^+}{|\S|}\S.
\end{align}
Function $\nu(e)$ represents the shear viscosity and $\tau_1(e), \tau_2(e)$ play the role of two activations\footnote{The required properties, like non-negativity and boundedness, are all summarized in Assumption~\ref{assonkoefs} on p.~\pageref{assonkoefs}.}. 
Behaviour of the fluid is Newtonian whenever $\tau_1(e) = \tau_2(e) = 0$. Considering $\tau_1(e)= 0$ and $\tau_2(e)> 0$, the fluid responds as the Bingham fluid and $\tau_2(e)$ is referred to as the yield stress. On the other hand, once $\tau_1(e) > 0$ and $\tau_2(e) = 0$, the fluid behaves as the Euler fluid until the shear rate $\tau_1(e)$ is reached. In order to eliminate the case when both $\tau_1$ and $\tau_2$ are positive, we impose an additional condition\footnote{Hence for any fixed $e$, we are always in the situation $\D\ve = \D\ve(\S)$ or $\S=\S(\D\ve)$; see also Figure~\ref{figura}.}
\begin{align} \label{sucin0tau}
\tau_1(e)  \tau_2(e)=0 \qquad \text{ for all } e \in \R.
\end{align}
For an easy non-trivial example that satisfies this as well as all other required properties given later in Assumption~\ref{assonkoefs}, we can take $\tau_1(e)\ddf \max\{0, \min \{1,e-1 \} \}$ and $\tau_2(e)\ddf \max\{0, \min \{1,1-e \} \}$.

Should one of the activating coefficients tend to infinity for a certain temperature, we would obtain the limit cases in the form of the inviscid fluid ($\tau_1 \to \infty$) or the rigid body motion ($\tau_2 \to \infty$), neither of which is covered in our analysis.
\begin{figure}[!tbp] 
  \centering
  \subfloat{\includegraphics[width=0.3\textwidth]{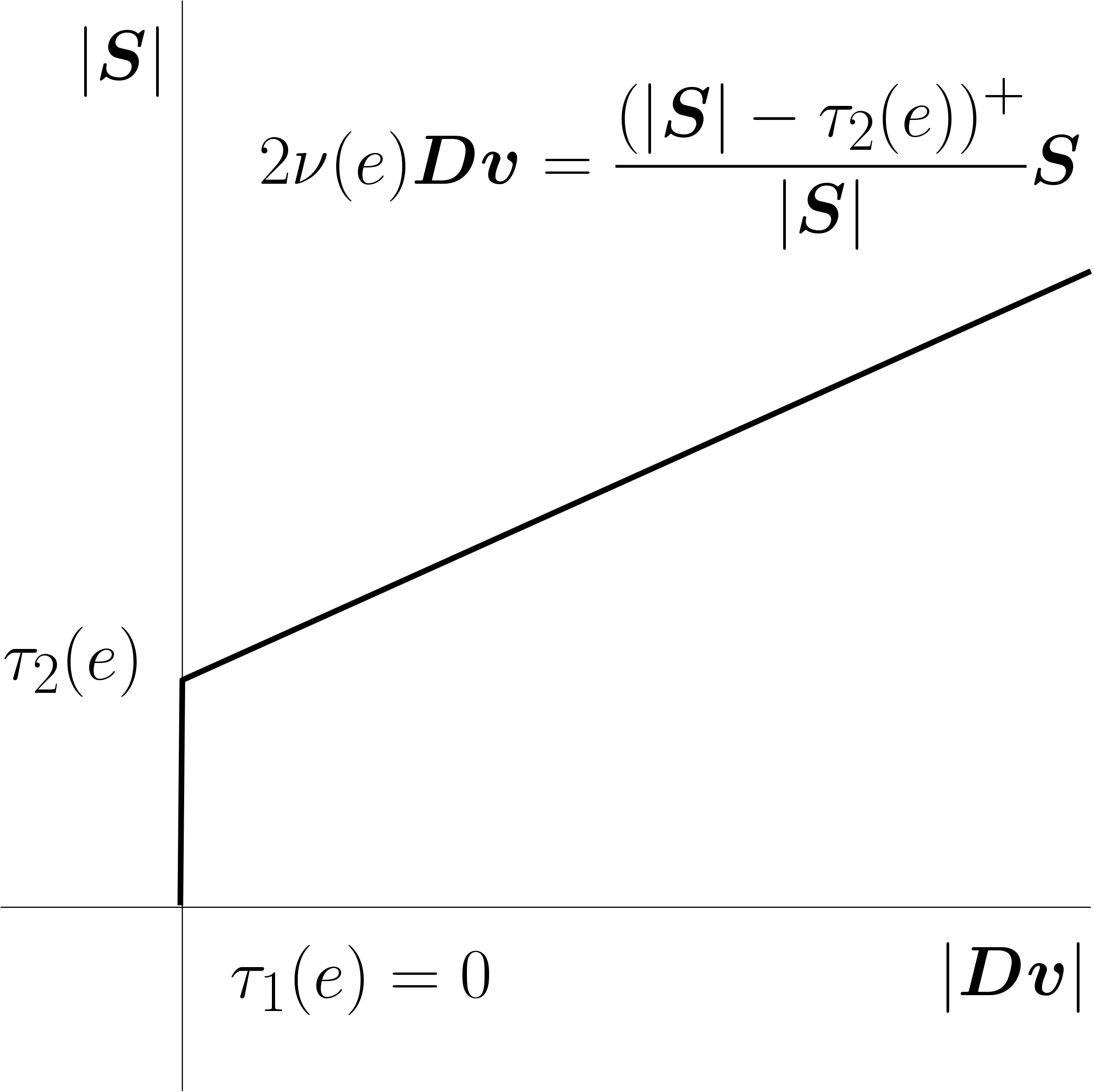}}
  \hfill
  \subfloat{\includegraphics[width=0.3\textwidth]{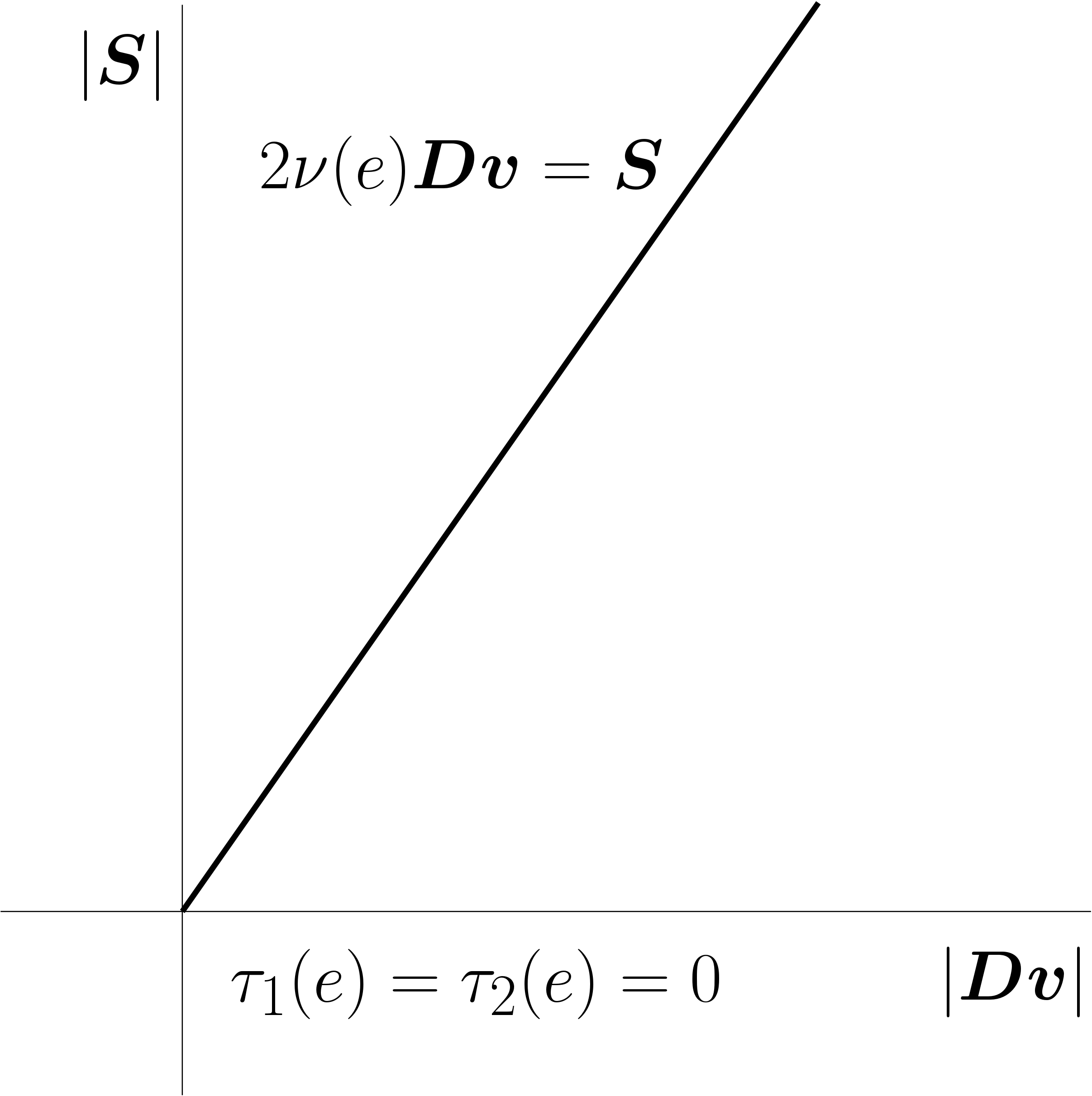}}
  \hfill
  \subfloat{\includegraphics[width=0.33\textwidth]{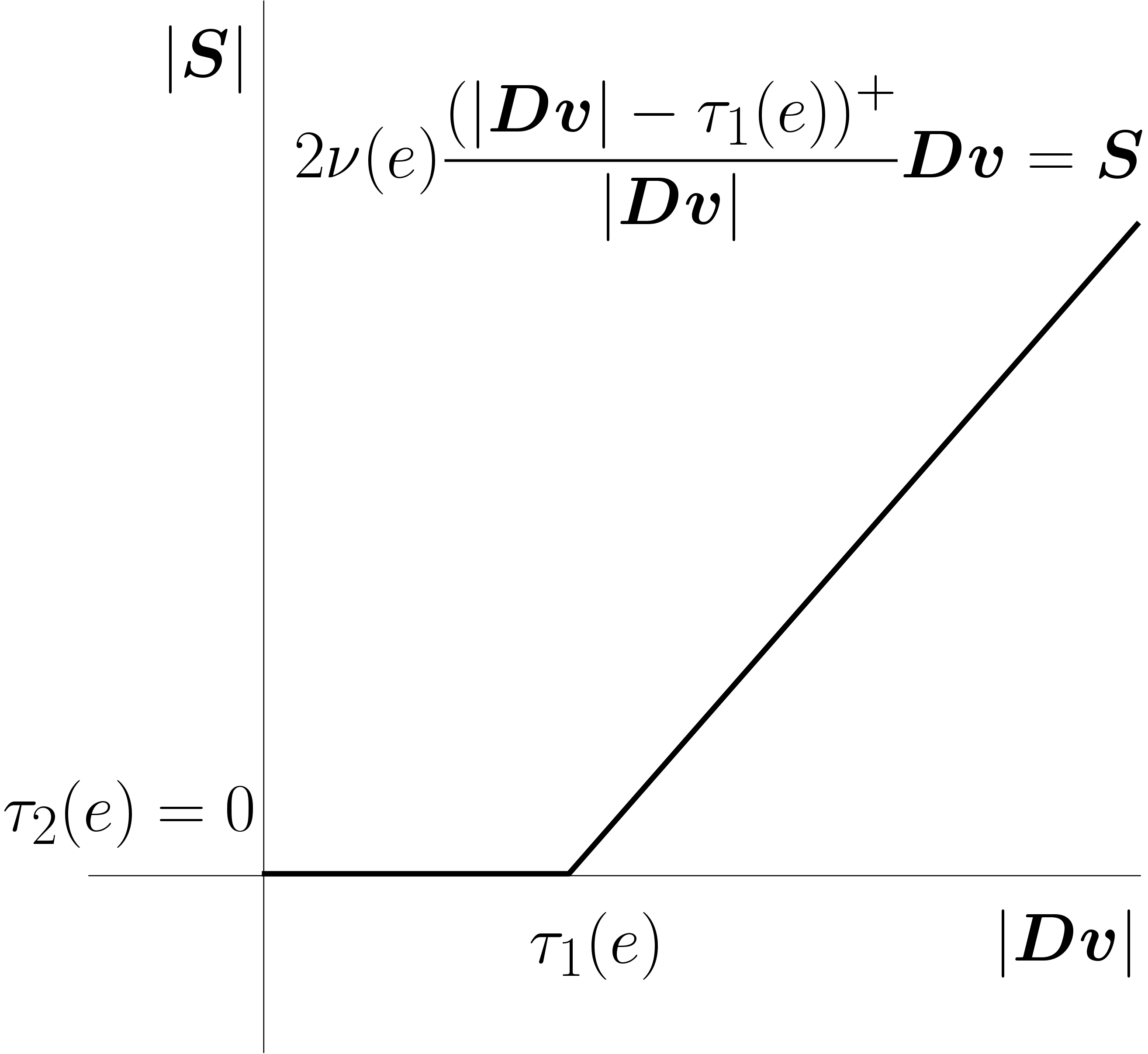}}
  \caption{Practically speaking, we consider these three cases and continuous transitions between them. Note that the slope of the slanted part may differ for different values of $e$.} \label{figura}
\end{figure}

Regarding the boundary conditions, numerous measurements (see below) have documented that not only non-Newtonian, but also Newtonian fluids can evince, in some situations, the so-called \emph{stick-slip boundary condition}. That is, the fluid adheres to the boundary until a certain critical value $\sigma>0$ of the wall shear stress is reached, at which moment the \emph{Navier slip} occurs. Denoting 
$$\s \ddf -(\S \n)_{\tau} \quad \text{on}~\Gamma,$$
the corresponding formula thus reads
\begin{align} \label{priklad}
\gamma \vt = \frac{(|\s|-\sigma)^+}{|\s|} \s \quad \text{on $\Gamma$},
\end{align}
for some \emph{friction coefficient} $\gamma > 0$. Not surprisingly, the critical value $\sigma$ and also $\gamma$ should depend on the temperature; see~\cite{drda2}. 
Many empirical models have been introduced over the years in pursuit of capturing this dependence; see e.g.~\cite{hatz} and references therein for models describing non-Newtonian fluids, in particular for polymer melts. In~\cite{granick}, the authors study the stick-slip boundary condition for Newtonian fluids. However, these models are usually not only very complicated, but also they often do not follow the real response of the fluid very precisely. We believe that the model introduced in our study is considerably simpler and yet more accurate, since the temperature-dependence is inherently imbedded in the activation-coefficient function. Moreover, due to the symmetric character of our boundary condition (see~\eqref{satis2}), the \emph{dual activation coefficient} is able to describe the opposite threshold-slip case, i.e.\ the \emph{perfect-slip--slip} situation. 

Similarly to the above constitutive relation~\eqref{satis1}, the quantities $\s$, $\vt$ and $e$ are coupled by means of a maximal monotone 2-graph parametrized by $e$. The triplet $(\s,\vt,e)$ satisfies the boundary condition if and only if $(\s,\vt,e) \in \B \subset \mathbb{R}^3 \times \mathbb{R}^3 \times \mathbb{R}$ a.e.\ on $\Gamma$, where graph $\B$ is defined via the implicit relation (cf.~\eqref{priklad})
\begin{align} \label{satis2}
(\s,\vt,e)\in\B \quad &\Longleftrightarrow\quad \ge \frac{\left(|\vt|-\sigma_1(e) \right)^+}{|\vt|} \vt= \frac{\left(|\s|-\sigma_2(e) \right)^+}{|\s|} \s.
\end{align}
Function $\gamma(e)$ represents the friction coefficient and functions $\sigma_1(e), \sigma_2(e)$ are another two activation coefficients. If $\sigma_1(e)= \sigma_2(e)= 0$, the Navier slip occurs. The case when only one of $\sigma_1(e), \sigma_2(e)$ is equal to zero represents threshold-slip boundary conditions: whenever $\sigma_1(e)= 0$ and $\sigma_2(e) > 0$, we observe the stick-slip boundary condition, while one calls the boundary condition perfect-slip--slip when $\sigma_1(e)> 0$ and $\sigma_2(e) = 0$. 
We assume that no other cases occur (cf.~\eqref{sucin0tau}), i.e. 
\begin{align} \label{sucin0sigma}
\sigma_1(e) \sigma_2(e)=0 \qquad\text{ for all } e \in \R.
\end{align}
Two limit cases that we are not able to incorporate are the \emph{genuine} perfect slip ($\sigma_1 \to \infty$) and the no-slip ($\sigma_2 \to \infty$) boundary conditions. On the other hand, even if we treated a case trivial with respect to~\eqref{sucin0tau}, i.e.\ $\S=2\ne\D\ve$, it is in general not known how to construct an integrable pressure when the no-slip boundary condition is considered (see~\cite{Frehse2003}). Hence we cannot say that our theory would miss an already solved case, for we always go to great pains to end up with an integrable pressure.


\paragraph{Result.}
Having introduced the model in its entirety, our mission is to develop the long-time and large-data existence theory. Relations \eqref{system1}--\eqref{system10} describe general flows of an incompressible heat-conducting fluid and its specific behavior is characterized by the relations \eqref{satis1} and \eqref{satis2}, determining the actual response of the fluid inside the domain (graph $\A$) and its interaction with the boundary (graph $\B$). A possible presence of the temperature-dependent activation criteria is included via the coefficient functions $\tau_1(e), \tau_2(e), \sigma_1(e)$ and $\sigma_2(e)$. These specify the flow regime as well as the behaviour on the boundary and also control the transition between the different states of the fluid.

We postpone the exact formulation of the existence result to Theorem~\ref{prime} on p.~\pageref{prime}, also because it uses specific notation that we want to introduce first. Suffice it to say that we are able to find a global weak solution, i.e.\ functions of certain regularity satisfying the balance equations in the form of integral identities or an inequality in the case of the relaxed balance of internal energy~\eqref{boie2}.

\paragraph{Discussion.}
The idea of implicit constitutive relations was introduced by Rajagopal in~\cite{rajag} and then thoroughly investigated by Bul\'{i}\v cek et al.\ in~\cite{BGMG, BGMRG} for general maximal monotone graphs (see also M\'alek~\cite{malek2008mathematical}), albeit without the heat effects. In~\cite{SS14}, Bul\'i\v{c}ek and M\'alek developed the existence theory for the stick-slip condition (i.e.~\eqref{priklad}) but still within the purview of isothermal processes only. Recently, the same authors in~\cite{MBthreshold} did enrich their original model with the heat conduction and a temperature-dependent viscosity. In our work, we consider this system further expanded, i.e.\ four thresholds are present in our model: $\tau_1, \tau_2, \sigma_1$ and $\sigma_2$, so that we can activate $\S, \D\ve, \s$ or $\vt$, respectively. In addition, we take into account the temperature-dependent friction $\gamma$ and viscosity $\nu$.

There are a number of works studying the temperature-dependent responses of fluid-like materials. Concerning our relation~\eqref{satis1} given by graph $\A$, it confers a unique possibility to trace the transfer from the yield stress ($\tau_1(e)=0$ and $\tau_2(e)>0$) through the Newtonian response ($\tau_1(e)=\tau_2(e)=0$) up to the inviscid/Euler regime, where the flow remains inviscid if $|\D\ve| \le \tau_1$. Quite remarkably, such temperature-driven viscous/inviscid transitions are not purely academic considerations but they are corroborated by experiments. In particular, one can encounter them in studies dealing with the boundary layers in the flow of almost inviscid fluids (see~\cite{dalibard} or~\cite{Schli}) or with flows of the superfluids, i.e.\ fluids that remain inviscid up to a certain temperature, above which the structure of the flow changes drastically; see~\cite{superfluids}. Furthermore, some fluids remain inviscid in the full range of experimentally observed data -- then we may set $\tau_1$ \emph{extremely large} (that is, beyond the range of experimental devices) to cover also such situations. Proceeding with the limit $\tau_1 \to \infty$ (if allowed) might also suggest a way of including possible \emph{solutions} to Euler equations without the pathological behaviour of those constructed in~\cite{MR3303948} and~\cite{DLSz}.
  
As far as~\eqref{satis2} is concerned, we point out the slip regime of the flow of polymer melts which is demonstrably temperature-dependent. It is studied and described in detail for example in~\cite{hatz}, where two mechanisms of slip are explained: the flow-induced chain detachment/desorption (weak slip) and the chain disentanglement (strong slip). Several models are also listed and referred to in~\cite{hatz}. For example, in~\cite{drda2}, there are presented experimental results focusing on molecular characteristics (especially on molecular weight) and temperature dependence of the stick-slip transition in a capillary flow and general features of high-stress rheology of a series of linear high-density polyethylene (HDPE) resins. The experiments are taken at different temperatures within the range 160--200~\degree C and HDPEs considered are of the molecular weights $M_w=130~500-316~600$. They describe the low-temperature anomaly, i.e.\ a situation in which the polymers deviate from the expected behaviour at the temperature of 160 \degree C. This phenomenon might arise from a flow-induced ordered phase along the wall at such lower temperatures when the wall shear stress reaches a certain critical value. Such a mesophase would introduce a local chain orientation at the PE/wall interface, so that it is then easier for the tethered chains to disentangle from the free chains. In other words, the interfacial chains need not be stretched (shear deformed) as much as when they are in the isotropic liquid phase, and the stick-slip transition would only require a lower critical stress. This serves as an example of a transition where the activation is dependent on the temperature. However, modifying the die wall condition allowed for an interesting observation -- instead of a smooth metallic wall, a rippled surface was used such that the melt in the 
thread valleys was stagnant and a PE/PE interface was achieved. For lower-molecular-weight resins, this PE/PE interface at the slip boundary apparently eliminates the mesophase formation and restores the normal features of the stick-slip transition, which means that the temperature anomaly at 160 \degree C completely disappears in the threaded die. The authors study also other nonstandard behaviour of the HDPEs, as well as different temperature ranges, which are considered e.g.\ in~\cite{drda3}. Moreover, as studied in \cite{granick}, the threshold slip occurs also in flows of Newtonian fluids, although it is not as remarkable as in the non-Newtonian case. But it is yet another viable example on the use of graph $\B$ all the same. For more information about various activations we want to draw attention to \emph{On Classification of Fluids with Activation. Part 1: Incompressible Fluids}, a paper currently in preparation authored by J.~Blechta, J.~M\'alek and K.~R.~Rajagopal.

\begin{figure}[h!]
\begin{center}
\includegraphics[width=0.4\textwidth]{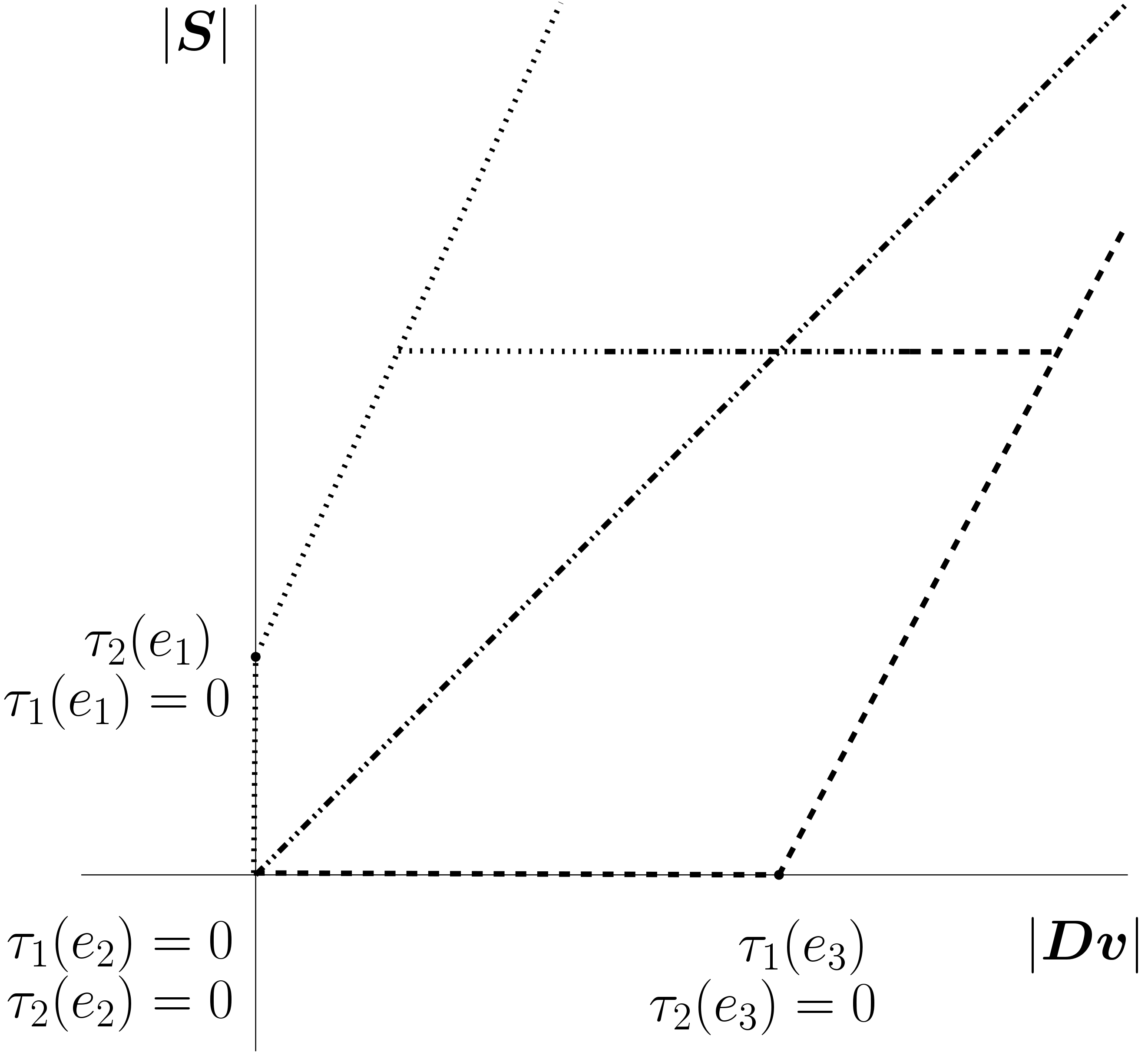}
\end{center}
\caption{An example of the hysteresis-like loop driven purely by the temperature change described by graph $\A$. Three most important situations are captured in the picture. Each dash style represents one constant value of the temperature -- the corresponding curve describes the behaviour of the fluid and its slope is ruled by the viscosity $\nu(e)$ at that temperature. The horizontal line where the styles alter indicates the continuous temperature change.}
\label{loop}
\end{figure}

Another advantage of temperature-dependent parameters is that the resulting structure allows to describe a hysteresis-like loop driven by changes in the temperature, which is one of the motivations for the use of the implicit theory
. Note that in Figure~\ref{loop}, one can see an example of such an application of such a model (as well as a prototypic example of such an hysteresis-like loop). 

Finally, we would like to emphasize that although this paper and its results deal only with non-linearities that can be described by maximal monotone 2-graphs, we are convinced that similarly to~\cite{bulivcek2009navier}, that served as the starting point for studying the complex mathematical theory for heat-conducting fluids in~\cite{BMR09}, the present paper will play the same role in the studies for more complex fluids with even more general classes of activaton criteria.


\paragraph{Paper arrangement.} In the following section, notation is established and assumptions on coefficients are posited. Section~\ref{result} contains the precise statement of our existence result. In section~\ref{prop}, coercivity and monotonicity of graphs $\A$ and $\B$ are shown, approximate graphs $\A^k$ and $\B^k$ are introduced and the fact that they \emph{do approximate} the original graphs is proved. Section~\ref{proofthm} then provides the actual proof of Theorem~\ref{prime}. The paper is concluded by Appendix where one can find a proof of an auxiliary lemma.

\section{Preliminaries}
For $0<t<T$ we write $Q_t\ddf[0,t) \times \om$ and $\Gamma_t \ddf [0,t) \times \partial \om$. As the introduction has given away already, we use simply $Q$ and $\Gamma$ for $Q_T$ and $\Gamma_T$, respectively. Continuous and compact embedding of two normed spaces is denoted by $\hookrightarrow$ and $\hookrightarrow\hookrightarrow$, respectively. No explicit distinction between spaces of scalar- and vector-valued functions will be made. Confusion should never come to pass as we employ small boldfaced letters to denote vectors and bold capitals for tensors. The same applies also to traces of Sobolev functions, which we denote like the original functions. Only when in need, we use $\tr$ for the trace.  For $1\le r\le\infty$ we denote $(L^r(\om), \|\! \cdot \! \|_r)$ and $(W^{1,r}(\om),\|\!\cdot\!\|_{1,r})$ the corresponding Lebesgue and Sobolev spaces. In particular, when an integral norm does not specify over which set it is taken, always $\om$ is implicitly considered. For a Banach space~$X$, the relevant Bochner space is designated by $L^r(0,T;X)$. 
When $r>1$, we set
\begin{align*}
W^{1,r}_{\n}(\om) & \ddf \bigl\{\f \in W^{1,r}(\om) \bigm| \tr \f \cdot \n = 0 \text{ on $ \partial \om$} \bigr\},\\
W^{-1,r}(\om) & \ddf \big( W^{1,\frac{r}{r-1}}(\om) \big)^*,\\
W^{-1,r}_{\n}(\om) & \ddf \big( W^{1,\frac{r}{r-1}}_{\n}(\om) \big)^*,\\
W^{1,r}_{\n,\dir}(\om) & \ddf \bigl\{\f \in W^{1,r}_{\n}(\om) \bigm| \dir \f = 0 \text{ in $\om$} \bigr\},\\
\C([0,T];X) &\ddf \{f \in L^{\infty}(0,T;X)  \bigm| 
 [0,T] \ni t^n \to t \Rightarrow f(t^n) \to f(t) \text{ strongly in $X$}\}, \\
\C_{w}([0,T];X) &\ddf \{f \in L^{\infty}(0,T;X)  \bigm| 
 [0,T] \ni t^n \to t \Rightarrow f(t^n) \to f(t) \text{ weakly in $X$}\}, \\
\C_c^1(\om) &\ddf \bigl\{f \in \C^1(\om) \bigm| \text{$f$ is compactly supported in $\om$} \bigr\}. 
\end{align*}
Sometimes, we will want to emphasize that certain regularity holds up to a given number, save this number, in the following way (see \cite{iwa} for the origin of the notation, even though here we thus define only sets of functions without any ambition to norm them):
\begin{align*}
L^{r)}(0,T;W^{1,r)}(\om)) &\ddf \bigcap_{1 \le s < r } \! L^s(0,T;W^{1,s}(\om)),
\intertext{and analogously the other way round, as in}
L^{(r}(0,T;W^{-1,(r}(\om))&\ddf \bigcup_{ s > r } L^s(0,T;W^{-1,s}(\om)),
\end{align*}
or in estimates (see e.g.~\eqref{ukazka}, \eqref{ukazka2} or~\eqref{ministr}). The positive and negative parts of a real-valued function $f$ are denoted $f^+ \ddf \max \{f,0 \}$ and $f^- \ddf \min \{f,0 \}$, respectively. With $\M(0,T)$ we denote the space of Radon measures on $(0,T)$. The symbol $\cdot$ stands for the scalar product of vectors or tensors, or simply for the product of fractions, depending on the context, and $\otimes$ signifies the tensor product. Generic constants are denoted by~$C$.

The external body forces $\b$ are for the sake of convenience supposed to be zero, i.e. 
\begin{align*} 
\b\ddf\0. 
\end{align*}

We close this preparatory section with requirements that we enforce upon the $e$-dependent coefficients in the definitions of graphs $\A$~\eqref{satis1} and $\B$~\eqref{satis2} and on the coefficient of heat conductivity $\kappa$ (see the introductory section for our motivation thereof):

\begin{ass} \label{assonkoefs}
The coefficients $\nu$, $\gamma$, $\kappa$, $\tau_1$, $\tau_2$, $\sigma_1$ and $\sigma_2$ are assumed to be continuous functions for which there exist $c_0,c_1,c_2 > 0$ such that
\begin{gather} \label{novost}
 0 \le \tau_1(s),\tau_2(s),\sigma_1(s),\sigma_2(s) \le c_0,  \\[4pt] \label{visco}
 c_1 \le \nu(s),\gamma(s),\kappa(s) \le c_2, \\[4pt] \label{eric}
0=\tau_1(s)\tau_2(s)=\sigma_1(s)\sigma_2(s),
\end{gather}
for all $s\in\R$.
\end{ass}
Indeed, there seems to be no logical reason for these coefficients to be defined outside of $\R^+$. We incorporate the possibility just for technical reasons because at the very first approximate level it may not be true that the internal energy is non-negative---the minimum principle begins to hold only after the first passing to limit (see~\cite{bulivcek2009navier})---and then we might conceivably deal with ill-defined arguments. Provided that we were presented with coefficients defined for non-negative values only, we would simply extend them by a constant to be defined everywhere.

\section{The result} \label{result}
Below we state the essence of this paper -- the existence theorem for our system. To get some idea regarding the professed regularity of individual quantities, see the uniform estimate~\eqref{ministr} and a rough sketch of its proof on p.~\pageref{prvniodhad}.

\begin{thm} \label{prime}
Let $T>0$, $\om \subset\R^3$ have a $\C^{1,1}$ boundary and Assumption~\ref{assonkoefs} stand valid. Consider initial conditions $\ve_0 \in \Ln(\om)$ and $e_0 \in L^1(\om)$, $e_0 \geq c_3$ a.e.\ in $\om$ for a certain positive constant $c_3>0$. Then there exists a weak solution to the system given by equations~\eqref{system1}--\eqref{system3}, \eqref{boie2}--\eqref{system10}, \eqref{satis1} and~\eqref{satis2}, i.e.\ a quintuplet $(\ve,p,\S,\s,e)$ fulfilling
\begin{align}
\ve&\in \C_w([0,T];L^2(\om)) \cap L^2(0,T;\Wnd(\om)), \nonumber\\
\dt \ve & \in L^{\frac{5}{3}}(0,T;W_{\n}^{-1,\frac53}(\om)), \nonumber\\
p &\in L^{\frac53}(Q), \nonumber\\
\S &\in L^2(Q), \nonumber\\
\s &\in L^2(\Gamma), \nonumber\\
e &\in \Li(0,T;L^1(\om)) \cap L^{\mfrac{5}{4}}(0,T;W^{1,\mfrac{5}{4}}(\om)),\nonumber\\
e &\geq c_3 \text{ a.e.\ in $Q$},\nonumber\\
\dt e &\in \M(0,T;W^{-1,\mfrac{10}{9}}(\om)),\nonumber\\
\dt E &\in L^{\mfrac{10}{9}}(0,T;W^{-1,\mfrac{10}{9}}(\om)), \label{primea}
\end{align}
having denoted $E \ddf \frac{|\ve|^2}{2} + e$. These functions satisfy for a.e.\ $t\in(0,T)$, all $\w \in \Wnn^{1,\frac52}(\om)$ and $u\in W^{1,\infty}(\om)$ the weak formulation
\begin{gather} \label{pprvni}
\langle \dt \ve,\w \rangle + \ii \Bigl( \S - p \I - \vov \Bigr)\cdot \nn \w \, dx + \iii \s \cdot \w \, dS = 0,  \\ \label{ddruha}
\langle \dt E,u \rangle + \ii \Bigl( \S\ve - (E+p)\ve + \ke\nn e \Bigr) \cdot \nn u \, dx + \iii \s \cdot \ve u \, dS = 0,
\end{gather}
where $(\S,\D\ve,e)$ and $(\s,\vt,e)$ satisfy~\eqref{satis1} and \eqref{satis2}, respectively.

In addition, we have the inequality
\begin{align} \label{ttreti}
\langle \dt e,u \rangle + \ii \bigl( -e \ve + \ke\nn e \bigr)\cdot \nn u \, dx &\geq \ii (\S \cdot\D\ve) u\, dx,
\end{align}
satisfied in the sense of measures for every non-negative $u\in W^{1,\infty}(\om)$.

The initial data are being attained in the form
\begin{align} \label{initc}
\lim_{t\to 0_+ }  \| \ve(t)-\ve_0\|_2+ \esslim_{t\to 0_+ }\|e(t)-e_0\|_1 =0.
\end{align}
\end{thm}
As a sidenote, our proof would work just fine for any $\b \in L^{2}(0,T;W^{-1,2}_{\n}(\om))$; one would only change regularity~\eqref{primea} to $\dt E \in L^1(0,T;W^{-1,\mfrac{10}{9}}(\om))$.

\section{Properties and approximations of the implicit relations} \label{prop}
Listing the qualities of graphs $\A$ and $\B$ crucial in the proof of Theorem~\ref{prime}, let us begin with an easy observation about their coercivity and monotonicity. Since the situation for $\B$ is by its definition quite analogous to that of $\A$, we will explicitly deal only with the latter:

\begin{lem}
Let graph $\A$ be of the form~\eqref{satis1} and Assumption~\ref{assonkoefs} hold. Then there exist $\alpha, \beta >0$ such that
\begin{gather}
\S \cdot \D \geq \alpha (|\S|^2+|\D|^2) - \beta, \label{koer} \\[4pt]
(\S - \S') \cdot (\D-\D') \geq 0, \label{mon} 
\end{gather}
 for any $(\S,\D,e), (\S',\D',e) \in \A$.
\end{lem}

\begin{proof}
Let us first consider $\tau_1(e)=0$, i.e. $$\D = \frac{1}{2\nu(e)} \cdot \frac{(|\S| - \ted)^+}{|\S|}\S.$$ 
Then, from~\eqref{novost} and~\eqref{visco},
\begin{equation*}
\S \cdot \D = \frac{|\S|}{2\nu(e)} (|\S| - \ted)^+ \geq \frac{1}{2\nu(e)} \left( (|\S| - \ted)^+ \right)^2 \geq \frac{1}{2 c_2} \left( (|\S| - c_0)^+ \right)^2 \geq \frac{1}{2 c_2} \left( \frac{|\S|^2}{4} - c_0^2\right).
\end{equation*}
Since also $$|\D|^2 = \frac{1}{(2\nu(e))^2} \left( (|\S| - \ted)^+ \right)^2,$$ we obtain simultaneously
\begin{equation*}
\S \cdot \D \geq 2\nu(e) |\D|^2 \geq 2 c_1 |\D|^2,
\end{equation*}
yielding \eqref{koer} in the case $\tau_1(e) =0$.

As for \eqref{mon}, the mapping $$ \Aa\longmapsto \frac{(|\Aa| - \ted)^+}{|\Aa|}\Aa, \qquad \Aa \in \R^{3\times 3},$$ is monotone and $\nu(e) > 0$, whence
\begin{equation*}
(\S - \S') \cdot (\D-\D')= \frac{1}{2\nu(e)} (\S - \S') \cdot \left( \frac{(|\S| - \ted)^+}{|\S|}\S - \frac{(|\S'| - \ted)^+}{|\S'|}\S' \right) \geq 0.
\end{equation*}
Should $\tau_1(e)>0$, then $\tau_2(e)=0$ by~\eqref{eric} and we would proceed completely analogously. 
\end{proof}

Identifying weak limits of non-linear quantities in the proof of Theorem~\ref{prime}, we will need the following convergence lemma, whose version without the internal energy has become quite standard (see~\cite{BGMG}). In our case, the addition of a compact sequence $(e^k)_k$ makes the situation somewhat trickier, however. Although stated for $\A$, the statement is equally valid also for $\B$. 

\begin{lem} \label{converlemma}
Let $\A$ be of the form~\eqref{satis1} and $U \subset Q$ be measurable. Consider sequences $(\S^k)_k$, $(\D^k)_k$ and $(e^k)_k$ of measurable functions on $U$, satisfying
 \begin{align*}
  &&(\S^k,\D^k,e^k) &\in \A &&\text{a.e.\ in $U$ for all $k\in\N$}, 
  \\
  && \S^k &\rightarrow \S && \text{weakly in $L^2(U)$}, 
  \\
  && \D^k &\rightarrow \D && \text{weakly in $L^{2}(U)$},
  \\
  && e^k &\rightarrow e && \text{a.e.\ in $U$}, 
  \\
  &&\limsup_{k \to \infty} \int_U \S^k \cdot \D^k &\le \int_U \S \cdot \D.&& 
 \end{align*}
Then $(\S,\D,e) \in \A$ a.e.\ in $U$ and $\S^k \cdot \D^k \to \S \cdot \D$ weakly in $L^1(U)$.
\end{lem}
\begin{proof}
This is a non-trivial modification of the convergence lemma given in~\cite{BGMG} with the first foray into the temperature-dependent territory in~\cite{MBthreshold}. The proof to Lemma~\ref{converlemma} can be found in Appendix, p.~\pageref{prvnilemma}.
\end{proof}

When proving Theorem~\ref{prime}, our starting point (one of them, actually; see the following section) is modifying graphs $\A$ and $\B$ to make them functional graphs as follows: Let $k \in \mathbb{N}$. Graph $\A$ will be approximated by $\A^k \subset \R^{3\times 3}\times\R^{3\times 3} \times \R$, defined through
\begin{align}\label{grafprvni}
(\S,\D,e)\in\A^k \quad \Longleftrightarrow \quad \S = \min \Big\{ k+2\ne, \frac{2\ne(|\D|-\tej)^++\ted}{|\D|} \Big\} \D.
\end{align}
where $(\S,\D,e)\in \R^{3\times 3}\times\R^{3\times 3} \times \R$. Note that in this approximation, $\S$ has become a continuous function of $(\D,e)$ and we define $$\S^k(\D,e) \ddf \S.$$ 
\begin{figure}[!tbp]
  \centering
  \subfloat{\includegraphics[width=0.53\textwidth]{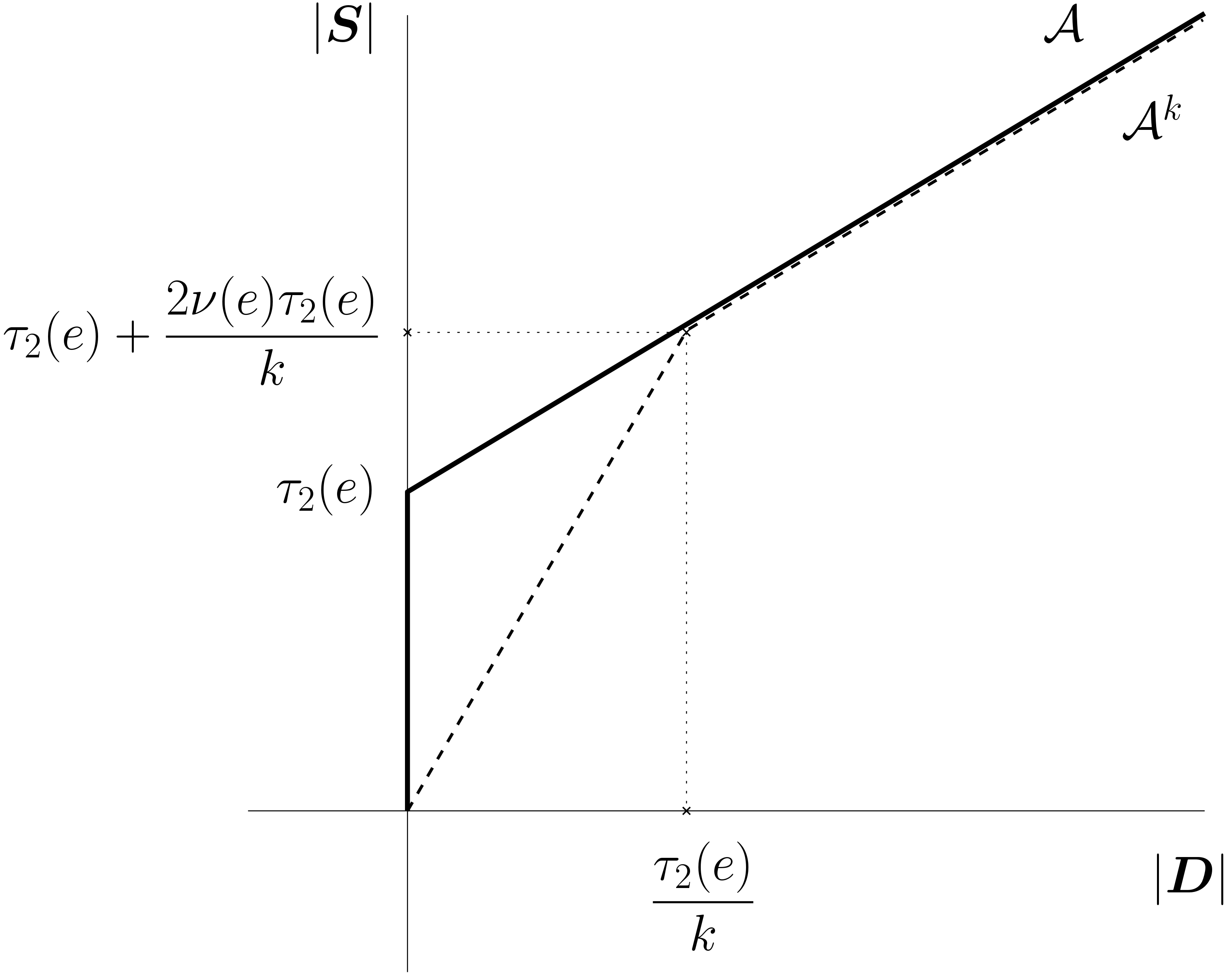}}
  \hfill
  \subfloat{\includegraphics[width=0.42\textwidth]{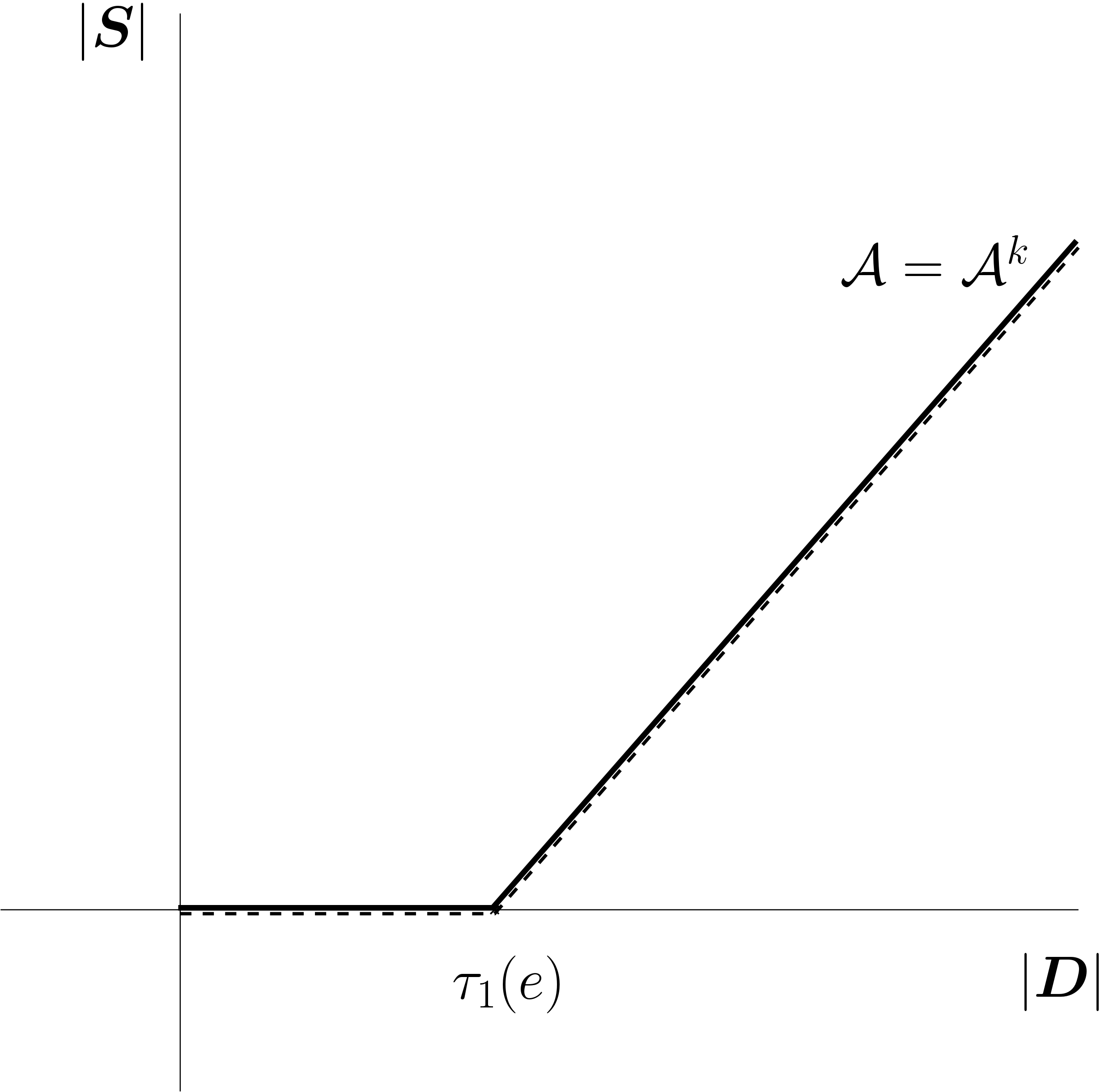}}
  \caption{Intuition behind $\A$ and $\A^k$ for $\ted >0$ and $\tej>0$, respectively. $\A^k$ is dashed.}
\end{figure}

Graph $\B$ will be modified quite analogously in $\big\{0\le |\vt| < \sged/k\big\}$ by means of $\B^k \subset \R^3\times\R^3 \times \R$, where
\begin{align*} 
 (\s,\vt,e)\in\B^k \quad \Longleftrightarrow\quad \s =  \min \Big\{ k+\ge, \frac{\ge(|\vt|-\sgej)^+ +\sged}{|\vt|} \Big\} \vt
\end{align*}
for $(\s,\vt,e)\in \R^3\times\R^3 \times \R$. Let us set $$\s^k(\vt,e)\ddf \s.$$

Both graphs $\A^k$ and $\B^k$ are still monotone and coercive, uniformly in $k$. We state this observation again for $\A^k$ only, the situation with $\B^k$ being completely analogous:
\begin{lem} \label{monok}
Let Assumption \ref{assonkoefs} hold. Then there exist $\alpha, \beta >0$ such that for all $k \in \mathbb{N}$ and arbitrary $(\S, \D, e), (\S', \D', e) \in \A^k$ it holds that 
\begin{gather*}
\S \cdot \D \geq \alpha (|\S|^2+|\D|^2) - \beta, \\[4pt]
(\S - \S') \cdot (\D-\D') \geq 0.
\end{gather*}
\end{lem}
\begin{proof}
This is obvious from definition~\eqref{grafprvni}.
\end{proof}

The approximated graphs prevent us from a straightforward employment of Lemma~\ref{converlemma}. Therefore we will need a slight modification of the result:

\begin{lem} \label{superconverlemma}
Let $U \subset Q$ be measurable, $\A$ be defined by~\eqref{satis1} and $\A^k$ be of the form~\eqref{grafprvni} for any $k\in \N$. Consider sequences $(\S^k)_k$, $(\D^k)_k$ and $(e^k)_k$ of measurable functions on $U$, satisfying
 \begin{align}
  &&(\S^k,\D^k,e^k) &\in \A^k &&\text{a.e.\ in $U$ for all $k\in\N$}, \nonumber
  \\
  && \S^k &\rightarrow \S && \text{weakly in $L^2(U)$}, \label{weakSk}\\
  && \D^k &\rightarrow \D && \text{weakly in $L^{2}(U)$},\nonumber
  \\
  && e^k &\rightarrow e && \text{a.e.\ in $U$}, \nonumber
  \\
  &&\limsup_{k \to \infty} \int_U \S^k \cdot \D^k &\le \int_U \S \cdot \D.&& \label{chivaldorek}
 \end{align}
Then $(\S,\D,e) \in \A$ a.e.\ in $U$ and $\S^k \cdot \D^k \to \S \cdot \D$ weakly in $L^1(U)$.
\end{lem}
\begin{proof}
This claim follows from Lemma~\ref{converlemma}, once we show that the approximated graphs $\A^k$ converge, in a sense, uniformly to $\A$. For $k \in \N$, consider the projection $$\Pk=(\Pk_{\S},\Pk_{\D},\Pk_e): \A^k \longrightarrow \A,$$ defined for $(\Aa,\BB,s) \in \A^k$ via 
\begin{align} \label{vlacil}
 \Pk(\Aa,\BB,s) \ddf 
\begin{cases}
 \Big(\Aa, \dfrac{1}{2\nu(s)}\cdot \dfrac{(|\Aa| - \tau_2(s))^+}{|\Aa|}\Aa,s\Big) \quad \text{if $0 \le |\Aa| < \tau_2(s) + \dfrac{2\nu(s)\tau_2(s)}{k}$},\\[10pt]
     \big( \Aa,\BB,s\big) \quad \text{else}.                                                           \end{cases}
\end{align}
 From definitions~\eqref{grafprvni} and~\eqref{vlacil} it follows that $$\big| \Pk_{\D}(\Aa,\BB,s) - \BB \big|  \le \dfrac{2 \tau_2(s)}{k}.$$
 \begin{figure}[!tbp]
  \centering
  \subfloat{\includegraphics[width=0.53\textwidth]{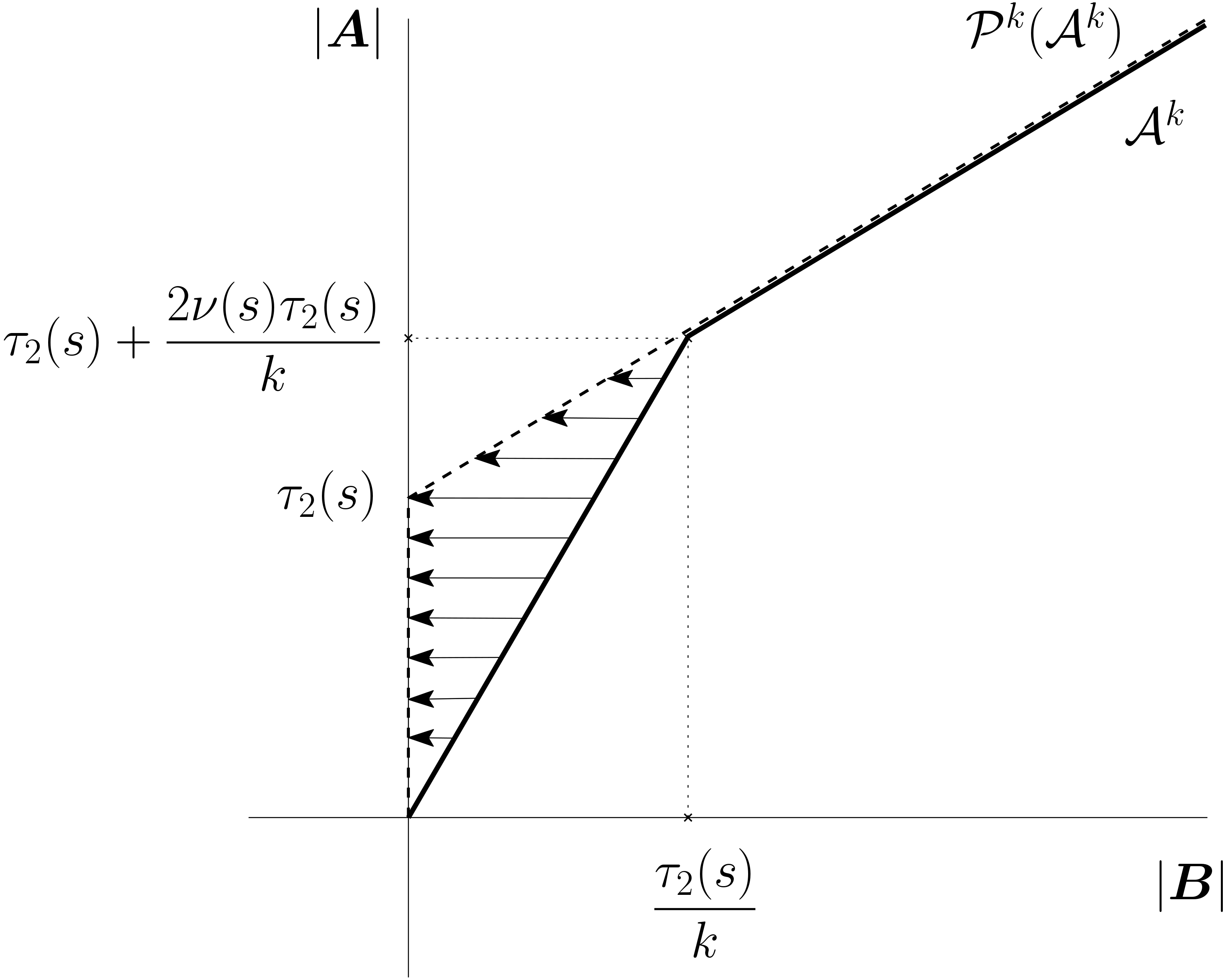}}
  \hfill
  \subfloat{\includegraphics[width=0.42\textwidth]{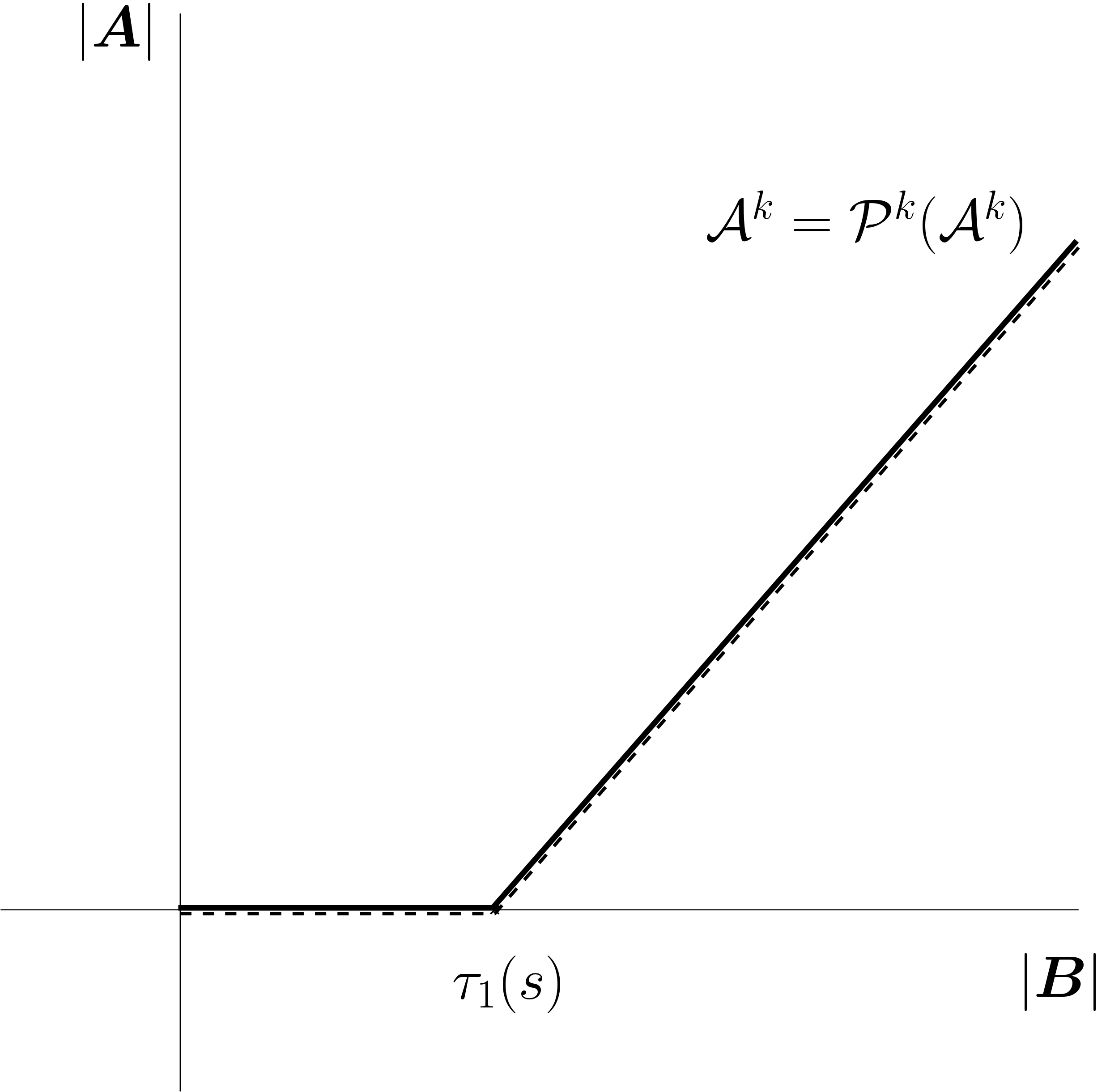}}
  \caption{Insight behind $\Pk$ for $\ted >0$ and $\tej>0$, respectively. $\Pk(\A^k)$ is dashed.}
\end{figure}
As a result, for every $(\Aa,\BB,s)\in\A^k$ it holds due to Assumption~\ref{assonkoefs} that 
\begin{align}\label{egoisthedonist}
\big\| \Pk(\Aa,\BB,s) - (\Aa,\BB,s) \big\|_{\R^{3\times 3}\times\R^{3\times 3} \times \R} \leq \dfrac{C}{k}, 
\end{align}
uniformly in $k>0$. Let us denote $\dk \ddf \Pk_{\D}(\S^k,\D^k,e^k)$. The functions $(\dk)_k$ are measurable on $U$ since $\Pk$ is continuous due to Assumption~\ref{assonkoefs}. The properties~\eqref{weakSk}--\eqref{egoisthedonist} thus enable us to infer  
\begin{align*}
  &&(\S^k,\dk,e^k) &\in \A &&\text{a.e.\ in $U$ for all $k\in\N$}, \\
  && \S^k &\rightarrow \S && \text{weakly in $L^2(U)$},\\
  && \dk &\rightarrow \D && \text{weakly in $L^{2}(U)$},\\
   && e^k &\rightarrow e && \text{a.e.\ in $U$}, \\
  &&\limsup_{k \to \infty} \int_U \S^k \cdot \dk &\le \int_U \S \cdot \D.&& 
 \end{align*}
Now both statements of the lemma follow from the already proven Lemma~\ref{converlemma}. The weak convergence in $L^1(U)$ is easily justified by $\D^k - \dk \rightarrow \0$ strongly in $L^{\infty}(U)$.
\end{proof}
\begin{rem}
Lemma~\ref{superconverlemma} holds completely analogously also for graph $\B$ and its approximations~$\B^k$. 
\end{rem}

\section{Proof of Theorem \ref{prime}} \label{proofthm}
The existence theorem is proved by means of a suitable approximation scheme indexed by parameter $k \in \N$ with which we will pass to infinity.

Firstly, the convective term is truncated so that the velocity field could be used as a test function in the balance of linear momentum, i.e.\ $\w \ddf \ve$ in \eqref{pprvni}. For this purpose, let $\Phi \in \C^1([0,\infty))$ be a non-increasing function such that $\Phi(x) = 1$ if $x \le 1$, $\Phi(x) = 0$ if $x\geq2$ and $\Phi(x)\in (0,1)$ otherwise, with $|\Phi'(x)| \le2$. For $k \in \mathbb{N}$ then define 
$$\Phi_k (x) \ddf \Phi(k^{-1}x). $$ 
Secondly, as foreshadowed already, the implicit relations given by $\A$ and $\B$ will be replaced by $\A^k$ and $\B^k$, respectively. Consequently, we will be enabled to use the standard Carath\'eodory theory to construct some stepping-stone solutions, since instead of the relatively complicated $\S$ and $\s$ in the original problem, we will deal with explicitly given $\Sk(\D\ve,e)$ and $\sk(\vt,e)\ddf -(\Sk(\D\ve,e)\n)_{\tau}$, respectively. Thirdly, due to the approximation's ability to take the velocity as a test function in the balance of linear momentum, there 
is no need for inequality \eqref{boie2} as the balance of total energy \eqref{system3} is equivalent to the balance of internal energy \eqref{boie}.

Recalling also our assumption to neglect the external forces, the original system will ergo be for $k \in \mathbb{N}$ approached by
\begin{align} \label{aprox1}
 \dt \ve + \dir( \vokk) - \dir \Sk(\D\ve,e) + \nn p &= \0 &&\text{in $Q$}, \\ \label{aprox2}
\dir \ve & = 0 &&\text{in $Q$},
 \\ \label{aprox3}
\dt e + \dir (e \ve - \ke \nn e) &= \Sk(\D\ve,e) \cdot \D\ve  &&\text{in $Q$}, \\ \label{aprox4}
\ve \cdot \n &= 0 &&\text{on $\Gamma$}, \\ \label{aprox5}
\vt  &= \sk(\vt,e) &&\text{on $\Gamma$}, \\ \label{aprox6}
\ke \nn e \cdot \n &= 0 &&\text{on $\Gamma$}, \\ \label{aprox7}
\ve( 0) &= \ve_0 &&\text{in $\Omega$},\\ \label{aprox8}
e( 0) &= e_0 &&\text{in $\Omega$}.
\end{align}
 That the above system possesses a weak solution can be expressed as follows:
\begin{lem} \label{sec}
Let assumptions of Theorem~\ref{prime} be met. Then for every $k \in \mathbb{N}$ there exists a weak solution to the approximated problem~\eqref{aprox1}--\eqref{aprox8}, i.e.\ a triplet $(\vk,\ek,\pk)$ such that
\begin{align}
\vk &\in \C([0,T];L^2(\om)) \cap L^2(0,T;\Wnd(\om)), \label{neunes}\\
\dt \vk & \in L^2(0,T;W^{-1,2}_{\n}(\om)),\nonumber \\
\pk &\in L^2(Q), \nonumber\\
\ek &\in \Li(0,T;L^1(\om)) \cap L^{\mfrac{5}{4}}(0,T;W^{1,\mfrac{5}{4}}(\om)), \label{ukazka}\\
\ek &\geq c_3 \text{ a.e.\ in $Q_T$}, \label{maxpr}
\\
\dt \ek &\in L^1(0,T;W^{-1,\mfrac{10}{9}}(\om)), \label{ukazka2}
\end{align}
satisfying for a.e.\ $t\in(0,T)$, all $\w \in \Wn(\om)$ and $u\in W^{1,\infty}(\om)$ the weak formulation
\begin{gather} \label{prvni}
\langle \dt \vk,\w \rangle + \ii \Bigl( \Sk - \pk \I - \vok \Bigr)\cdot \nn \w \, dx + \iii \sk \cdot \w \, dS = 0,  \\ \label{druha}
\langle \dt \ek,u \rangle + \ii \bigl( -\ek\vk + \kk\nn\ek \bigr)\cdot \nn u \, dx = \ii (\Sk\cdot\D\vk)u\, dx,
\end{gather}
abbreviating $\Sk\ddf \Sk(\D\vk,e^k)$ and $\sk\ddf\sk(\vt^k,e^k)$. In addition, the internal energy satisfies for a.e.\ $t\in(0,T)$ and any non-negative $u\in W^{1,\infty}(\om)$
\begin{align} \label{treti}
\ii \sqrt{\ek(t)} \, u \, dx - \ii \sqrt{e_0} \,u \, dx   + \iqt \Bigl( -\sqrt{\ek}\, \vk + \dfrac{\kk \nn\ek}{2 \sqrt{\ek}} \Bigr)\cdot \nn u \, dx \, dt \geq 0. 
\end{align}
The initial conditions are attained through
\begin{align*}
\lim_{t\to 0_+ }  \| \vk(t)-\ve_0\|_2+ \esslim_{t\to 0_+ }\|\ek(t)-e_0\|_1=0.
\end{align*}
Furthermore, we have the estimate
\begin{align} \label{ministr}
\esssup_{t\in(0,T)}\Bigl( \|\vk(t)\|_2^2 + \|\ek(t)\|_1 \Bigr) + \itt \Bigl( \|\vk\|_{1,2}^2+\|\ek \|_{1,\mfrac{5}{4}}^{\mfrac{5}{4}}+\|\pk\|_{\frac{5}{3}}^{\frac{5}{3}} + \|\Sk\|_2^2 \Bigr) \, dt + \ig|\sk|^2 \, dS\, dt \le C.
\end{align}
where $C$ is independent of $k$.
\end{lem}
\begin{rem}\label{mministry} Notice that Lemma~\ref{sec} implies straightaway
\begin{align} \label{ministry}
\itt \Bigl( \|\vk\|_{\frac{10}{3}}^{\frac{10}{3}}+\|\ek \|_{\mfrac{5}{3}}^{\mfrac{5}{3}} + \| \dt \vk \|_{W_{\n}^{-1,\frac53}}^{\frac53} + \| \dt \ek \|_{W^{-1,\mfrac{10}{9}}}^{\phantom{\frac53}} \Bigr) \, dt + \ig|\vk|^{\frac{8}{3}} \, dS\, dt \le C,
\end{align}
with $C$ independent of $k$.
\end{rem}
\begin{proof}[Proof of Remark \ref{mministry}] The first two terms follow from the embeddings 
\begin{gather*}
\Li(0,T;L^2(\om))\cap L^2(0,T;W^{1,2}(\om)) \hookrightarrow L^{\frac{10}{3}}(0,T;L^{\frac{10}{3}}(\om)), \\
\Li(0,T;L^1(\om))\cap L^{\mfrac{5}{4}}(0,T;W^{1,\mfrac{5}{4}}(\om)) \hookrightarrow L^{\mfrac{5}{3}}(0,T;L^{\mfrac{5}{3}}(\om)),
\end{gather*}
respectively. For the boundary term we add also the trace theorem via $$\iii |\vk|^{\frac{8}{3}}\,dS \le C \|\vk\|^{\frac{8}{3}}_{\frac{3}{4},2}\le C\|\vk\|^{\frac{2}{3}}_2\|\vk\|^2_{1,2}. $$
The bounds on the time derivatives follow from the weak formulations \eqref{prvni} and \eqref{druha}, respectively, estimate~\eqref{ministr} and the already proven part of~\eqref{ministry}.
\end{proof}

\begin{proof}[Notes on the proof of Lemma \ref{sec}]
Towards proving Lemma \ref{sec}, the problem~\eqref{aprox1}--\eqref{aprox8} would be approximated even further by a \emph{quasicompressible} system with a parameter $\e>0$; our plan being to take the limit $\e\to0_+$. Specifically, the solenoidal condition \eqref{aprox2} is replaced by the Neumann problem for the pressure, so that the approximation could be symbolically described as
\begin{align} \label{kejepr}
\phantom{\ii} \dir \ve & = 0 \quad \text{in $Q$}\quad \xrsquigarrow{\qquad \quad} \quad
\left\{ \
\begin{aligned}
 \phantom{\ii} \dir \ve & = \e \Delta p \phantom{0} \quad \text{in $Q$},  \\
\nn p \cdot \n &= 0 \phantom{\e \Delta p} \quad \text{on $\Gamma$}, \\
\ii p(x)\, dx &= 0 \phantom{\e \Delta p} \quad \text{in $(0,T)$}.
\end{aligned}
\right.
\end{align}
The quasicompressible approximation helps us construct an integrable pressure and once we have a solution thereof, the limit pass $\e \to 0_+$ to~\eqref{aprox1}--\eqref{aprox8} is achieved by means of $\e$-uniform estimates. These are available thanks to the Helmholtz decomposition of the space $\Wn(\om)$ (which does not apply to $W^{1,2}_0(\om)$, for instance). A solution to the quasicompressible approximation can be found by dint of a two-level Galerkin space discretization. This is necessary e.g.\ for reasons of warranting the minimum principle for the internal energy~\eqref{maxpr}.

Due to its having been done in~\cite{bulivcek2009navier} for a similar situation, we will skip the proof. Compared to the proof of our main Theorem~\ref{prime}, the differences between the proof of Lemma~\ref{sec} and its counterpart in~\cite{bulivcek2009navier} are practically aesthetic only: Firstly, $e$-dependent coefficients, whose convergence is easily deduced by compactness of the internal energy both in $Q$ and on $\Gamma$, and Assumption~\ref{assonkoefs} so that they add no additional difficulty whatsoever. Secondly, presence of $\A^k$ and $\B^k$ and necessity for convergence of pertinent non-linear terms are equally simple by means of Minty's trick (recall both $\A^k$ and $\B^k$ are continuous and Lemma~\ref{monok} guarantees monotonicity). 

However, not wanting to rest upon the blackbox of~\cite{bulivcek2009navier} completely, let us, at least informally, derive inequality~\eqref{treti} and the uniform estimate~\eqref{ministr} here: The inequality~\eqref{treti} plays only an auxiliary role for the attainment of the initial value of the internal energy. If we could test the balance of internal energy~\eqref{druha} with the function
\begin{align*}
\overline{u}(t,x) \ddf \dfrac{u}{2\sqrt{\ek}}\chi_{(0,t)},
\end{align*}
where $\chi_{(0,t)}$ is the characteristic function of the interval $(0,t)$ and $u\in W^{1,\infty}(\om)$ is arbitrary yet non-negative, we would end up with
\begin{multline} 
\ii \! \sqrt{\ek(t)} \, u \, dx - \ii \! \sqrt{e_0} \,u \, dx   + \iqt \! \Bigl( -\sqrt{\ek}\, \vk + \dfrac{\kk \nn\ek}{2 \sqrt{\ek}} \Bigr)\cdot \nn u \, dx \, dt \\ = \iqt\! \Big( \dfrac{\Sk\cdot\D\vk}{2\sqrt{\ek}} + \dfrac{\kk |\nn \ek|^2}{4\ek\sqrt{\ek}}\Big) u\, dx \, dt \geq 0. \label{brandt}
\end{multline} 
Indeed, this is possible only formally, since with such a choice we no longer have $\overline{u}(t) \in W^{1,\infty}(\om)$, as required in~\eqref{druha}. This problem, however, can be overcome in multiple ways, e.g.\ by testing the equation at the level of sufficiently regular approximations and then showing that the inequality~\eqref{brandt}, being already well-defined under the properties~\eqref{neunes}--\eqref{maxpr}, is stable under passing to limit. We will corroborate this stability for the final limit $k\to \infty$ when justifying attainment of the initial internal energy on p.~\pageref{intene}, which will also edify about stability leading to~\eqref{treti} in the first place.  

Concerning the estimate~\eqref{ministr}, for a fixed $k \in \mathbb{N}$, the velocity field $\vk$ remains an admissible test function in the balance of linear momentum~\eqref{prvni}. Hence the energy equality in the balance of linear momentum holds, and when combined with coercivity from Lemma~\ref{monok} and Korn's inequality, we deduce
\begin{align} \label{prvniodhad}
\esssup_{t\in(0,T)} \|\vk(t)\|_2^2  + \itt \Big( \|\vk\|_{1,2}^2 + \| \Sk \|_2^2 \Big) \, dt + \ig|\sk|^2 \, dS\, dt \le C.
\end{align}
For the estimate on the pressure, we consider an auxiliary problem \begin{align*} 
-\Delta h &= |\pk|^{-\frac13}\pk - \frac{1}{|\om|}\int_{\om} |\pk|^{-\frac13}\pk\, dx \quad \text{in $Q$,} \\
\nn h \cdot \n &=0 \quad \text{on $\Gamma$,} \\
\int_{\om} h \, dx & =0 \quad \text{in $(0,T)$.}
\end{align*}
Since the elliptic theory for Neumann's problem (here we need the $\C^{1,1}$ boundary of $\om$; see e.g.~\cite{MR775683}) yields the bound 
$$\|\nn h \|_{1,\frac52} \leq C \| \pk \|_{\frac53} $$
with $C$ influenced by $\om$ only, and~\eqref{kejepr} guarantees $$- \iq \pk \I \cdot \nn^2 h \, dx \, dt = \int_0^T \|\pk \|_{\frac53}^{\frac53} \, dt,$$
we take $\w \ddf \nn h \in \Wn(\om)$ in~\eqref{prvni} and hence (after some computation, exploiting particularly mutual orthogonality of $\w$ and $\dt\vk$) derive 
\begin{align*} 
\itt \|\pk \|_{\frac53}^{\frac53} \, dt \le C.
\end{align*}
Estimates for equations with the right-hand side in $L^1$ yield the rest: in~\eqref{druha}, let us first take $u\ddf 1$ and then $u\ddf (\ek)^{\lambda}$ for $-1<\lambda<0$. The second choice of the test function yields
 $$\iq\frac{|\nn \ek|^2}{(\ek)^{1-\lambda}} \leq \frac{C}{\lambda+1} $$
and thence altogether we can deduce
\begin{align*}
\esssup_{t\in(0,T)} \| e^k(t) \|_1 + \itt \|\ek \|_{1,\mfrac{5}{4}}^{\mfrac{5}{4}} \, dt \le C.
\end{align*}
See~\cite{bulivcek2009navier} for details. 
\end{proof}
Having already proved the additional estimate~\eqref{ministry}, the real work has to be done only in the final limit, i.e.\ $k\to\infty$, which is exactly what the proof of Theorem~\ref{prime} is all about.

\begin{proof}[Proof of Theorem \ref{prime}.]
 Uniform estimates~\eqref{ministr} and~\eqref{ministry} combined with the Banach-Alaoglu, the Eberlein-\v Smulian and the Aubin-Lions theorems allow us to assume that as $k\to\infty$, the following convergences take place:
\begin{align}
\label{ka} &&\vk &\rightarrow \ve  &&\text{weakly in $L^2(0,T;\Wnd(\om))$}, \\
\label{kb} &&\vk &\rightarrow \ve  &&\text{weakly$^*$ in $\Li(0,T;L^2(\om))$},\\
\label{kc} &&\vk &\rightarrow \ve  &&\text{strongly in $L^{\mfrac{10}{3}}(Q)$}, \\
\label{ke} &&\dt \vk &\rightarrow \dt \ve  &&\text{weakly in $L^{\frac{5}{3}}(0,T;W_{\n}^{-1,\frac53}(\om))$},\\
\nonumber
&&\pk &\rightarrow p  &&\text{weakly in $L^{\frac{5}{3}}(Q)$},\\
\label{kg} &&\ek &\rightarrow e  &&\text{weakly in $L^{\mfrac{5}{4}}(0,T;W^{1,\mfrac{5}{4}}(\om))$}, \\
\label{kh} &&\ek &\rightarrow e  &&\text{strongly in $L^{\mfrac{5}{3}}(Q)$}, \\
\label{kl} &&\ek &\rightarrow e  &&\text{a.e.\ in $Q$}, \\
\label{ki} &&\dt \ek &\rightarrow \dt e  &&\text{weakly in $\M(0,T;W^{-1,\mfrac{10}{9}})$}, \\
\label{kk} &&\Sk &\rightarrow \S &&\text{weakly in $L^2(Q)$},\\
\label{kj} &&\sk &\rightarrow \s &&\text{weakly in $L^2(\Gamma)$}.\\
\intertext{We are also able to derive compactness of the traces of $\vk$ and $\ek$, namely that they satisfy}
\label{kd} &&\vk &\rightarrow \ve  &&\text{strongly in $L^2(\Gamma)$}, \\
\label{kn}
&&\ek &\rightarrow e  &&\text{strongly in $L^1(\Gamma)$}.
\end{align}
Indeed, the strong convergences~\eqref{kc} and~\eqref{kh}, interpolation and the trace theorem imply
\begin{gather*}
L^2(0,T;\Wn(\om)) \cap W^{1,\frac{5}{3}}(0,T;W_{\n}^{-1,\frac53}(\om)) \hookrightarrow\hookrightarrow L^2(0,T;W^{1),2}_{\n}(\om)) \hookrightarrow L^2(\Gamma), \\
L^{\mfrac{5}{4}}(0,T;W^{1,\mfrac{5}{4}}(\om))
\cap W^{1,1}(0,T;W^{-1,\mfrac{10}{9}}(\om)) \hookrightarrow\hookrightarrow L^{1}(0,T;W^{\alpha,\frac54}(\om)) \hookrightarrow L^1(\Gamma)
\end{gather*}
for any $\alpha\in(4/5,1)$ as the first embedding on the second line holds for any $0<\alpha<1$, while the second one only for $\alpha > 4/5$. All in all,~\eqref{kd} and~\eqref{kn} are thus justified. Notice also that~\eqref{maxpr}, \eqref{ministr}, \eqref{kl} and Fatou's lemma imply
$$e \in L^{\infty}(0,T;L^1(\om)). $$
\paragraph{Weak formulation.}
The above convergences make passing from~\eqref{prvni} to~\eqref{pprvni} trivial. The situation in~\eqref{druha} is quite the opposite and we are unable to pass to the limit in its current form. Therefore we replace the equation for the internal energy $\ek$ with that for the total energy $$\Ek \ddf \frac{|\vk|^2}{2}+\ek.$$
To this end, let $u\in W^{1,\infty}(\om)$ and in~\eqref{prvni} take $\w \ddf \vk u \in \Wn(\om)$. Adding the resulting identity to~\eqref{druha}, for a.e.\ $t\in(0,T)$ we obtain
\begin{align} \label{altap}
\langle \dt \Ek,u \rangle + \ii \Bigl( \Sk\vk - (\Ek+\pk)\vk + \kk\nn\ek \Bigr) \cdot \nn u \, dx + \iii \sk \cdot \vk u \, dS +\f^k(u) &= 0,
\end{align}
where, using $\dir \vk = 0$,
\begin{align*}
 \f^k(u) & = \ii \left( \frac{|\vk|^2}{2}\,\vk \cdot \nn u - \vok \cdot (u \nn \vk + \vk \otimes \nn u) \right) dx \\
& = \ii \left( \frac{|\vk|^2}{2}\,\vk \cdot \nn u - u \vk \cdot \nn \Big( \int_{0}^{|\vk|} s \Phi_k(s)\, ds \Big)  - \Phi_k(|\vk|) |\vk|^2 \vk \cdot \nn u \right) dx \\
& = \ii \left( \frac{|\vk|^2}{2}\,\vk \cdot \nn u +\Big( \int_{0}^{|\vk|} s \Phi_k(s)\, ds \Big) \vk \cdot \nn u  - \Phi_k(|\vk|) |\vk|^2 \vk \cdot \nn u \right) dx \\
& = \ii \left( \int_{0}^{|\vk|} s (\Phi_k(s)-1) \, ds + (1- \Phi_k(|\vk|)) |\vk|^2 \right) \vk \cdot \nn u \, dx
\end{align*}
and we notice
\begin{align} \label{hazel}
 \itt |\f^k(u)| \, dt \to 0
 \end{align}
for $k\rightarrow\infty$ and actually every $u \in L^{(10}(0,T;W^{1,(10}(\om))$ due to~\eqref{kc} and the fact that $\Phi_k \nearrow 1$. Taking into consideration~\eqref{ka}--\eqref{kj}, it is hence easy to pass from~\eqref{altap} to~\eqref{ddruha}. Note that the estimates~\eqref{ministr} and~\eqref{ministry} yield 
\begin{align} \label{wysiwyg}
\sup_k \Big( \big\|\Sk\vk - (\Ek+\pk)\vk + \kk\nn\ek \big\|_{L^{\mfrac{10}{9}}(Q)} + \| \sk \cdot \vk \|_{L^{\frac{10}{9}}(\Gamma)} \Big) < \infty,
\end{align}
so that unlike~\eqref{ki}, for the total energy we can suppose by~\eqref{altap}, \eqref{hazel} and~\eqref{wysiwyg}, 
\begin{align*}
\dt  \Ek \to \dt E \quad \text{weakly in $L^{\mfrac{10}{9}}(0,T;W^{-1,\mfrac{10}{9}}(\om))$},
\end{align*}
yielding~\eqref{primea}.

\paragraph{Threshold conditions.}
Next we will show validity of~\eqref{satis1} and~\eqref{satis2}, i.e.\ practically that the weak limit of a non-linear function is this non-linear function of the weak limit of its argument. For this sake we employ Lemma~\ref{superconverlemma}. Regarding~\eqref{satis2}, all prerequisites of the convergence lemma are met, including the $\limsup$ condition~\eqref{chivaldorek} due to the compactness of traces. Hence $(s,\vt,e)\in\B$ a.e.\ on $\Gamma$ on account of~\eqref{kj}--\eqref{kn}.

Towards the other inclusion, i.e.\ $(\S,\D \ve,e)\in\A$ a.e.\ in $Q$, the crucial $\limsup$ condition is still to be proved:
\begin{align} \label{budesehodit}
\intej \S\cdot\D\ve \geq 
\limsup_{k\to \infty} \intej \Sk\cdot \D \vk,
\end{align}
for some measurable $E^j\subset Q$, $j\in\N$, satisfying $\displaystyle\lim_{j\to\infty} |Q\setminus E^j|=0$. The procedure to achieve this end is based on the $\Li$-truncation method~\cite{MR1183665,MR1602949}, as we cannot take $\w\ddf\ve$ in the weak formulation~\eqref{pprvni}. Its deployment necessitates the pressure decomposition first.

\paragraph{Pressure decomposition.}
Although we have a uniform $\nicefrac{5}{3}$-integrability result for the pressure in~\eqref{ministr}, this is not yet satisfactory. For the ensuing $\Li$-truncation method we need the pressure to be at least $L^2(Q)$ or, on the other hand, Sobolev in $\om$, albeit with an exponent of integrability less than $\nicefrac{5}{3}$. In this paragraph we will show that we can split $$\pk = \pk_1 + \pk_2$$ with $\pk_1 \in L^2(Q)$ and $\pk_2 \in L^{\frac54}(0,T;W^{1,\frac54}(\om))$; with the respective bounds in both cases uniform in $k$ (for this step we need a $\C^{1,1}$ boundary of $\om$), implying that we can assume
\begin{align} \label{redspa}
\pk_1 &\rightarrow p_1 \quad \text{weakly in $L^2(Q)$,}\\ \label{redspb}
\pk_2 &\rightarrow p_2  \quad \text{weakly in $L^{\frac54}(0,T;W^{1,\frac54}(\om))$,}
\end{align}
with $p_1+p_2=p$. We will only sketch the procedure as the detailed version is to be found in~\cite{ja4} and the very origin lies in~\cite{BMR09}.

According to~\eqref{prvni}, for any $\varphi \in W^{2,2}(\om)$ such that $\nn \varphi \cdot \n = 0$ at $\partial \om$ and a.e.\ $t \in (0,T)$, we have
\begin{align*} 
 \int_{\om}\pk \Delta \varphi\, dx = \int_{\om } \big (\Sk - \vok \big) \cdot \nn^2 \varphi \, dx + \int_{\partial \om } \sk \cdot \nn \varphi \, dS = 0.
\end{align*}
Let the partial pressure $\pk_1$ be for a.e.\ $t\in(0,T)$ given as the unique weak solution to
\begin{align*} 
 \int_{\om}\pk_1 \Delta \varphi\, dx &= \int_{\om } \Sk \cdot \nn^2 \varphi \, dx+ \int_{\partial \om } \sk \cdot \nn \varphi \, dS, \\
 \nn \pk_1 \cdot \n &= 0 \quad \text{on $\partial \om$}, \\ \int_{\om}\pk_1 &= 0, 
 \intertext{and $\pk_2 \ddf \pk - \pk_1$ so that~$\pk_2$ satisfies for a.e.\ $t\in(0,T)$}
 - \int_{\om}\nn \pk_2 \cdot \nn \varphi\, dx &= \int_{\om } \dir \big( \vok \big) \cdot \nn \varphi \, dx,
 \end{align*}
in both cases for all $\varphi \in W^{2,2}(\om)$ such that $\nn \varphi \cdot \n = 0$ at $\partial \om$. The elliptic theory then makes it possible to derive $k$-uniform bounds leading to~\eqref{redspa} and~\eqref{redspb} (note that the sequence $\big(\dir ( \vok )\big)_k$ is bounded in $L^{\frac54}(Q)$). 

Now we can bound $(\dt \vk)_k$ in $L^2(0,T;W_{\n}^{-1,2}(\om)) + L^{\frac54}(Q)$: Let $\w \in L^2(0,T;\Wn(\om)) \cap L^5(Q)$. Weak formulation~\eqref{prvni} and the above pressure decomposition yield $\dt \vk = \dt \vk_1 + \dt \vk_2$, where
\begin{align*}
\langle \dt \vk_1,\w \rangle &\ddf \ii \Bigl( \Sk - \pk_1 \I \Bigr)\cdot \nn \w \, dx + \iii \sk \cdot \w \, dS, \\
\langle \dt \vk_2,\w \rangle &\ddf \ii \Bigl( \dir \vok + \nn \pk_2 \Bigr)\cdot \w \, dx,
\end{align*}
so that, due also in part to estimate~\eqref{ministr},
\begin{align*}
\int_0^T \big| \langle \dt \vk_1,\w \rangle \big| &\le  C  \| \w \|_{L^{2}(\Wn)}, \\ 
\int_0^T \big| \langle \dt \vk_2,\w \rangle \big| &\le  C  \| \w \|_{L^5(Q)},
\end{align*}
uniformly in $k$. As such, one may suppose 
\begin{align} \label{redspc}
\dt \vk \rightarrow \dt \ve \quad \text{weakly in $L^2(0,T;W_{\n}^{-1,2}(\om)) + L^{\frac54}(Q)$.}
\end{align}

\paragraph{Preliminary definitions.}
Let $N\in\N$ be arbitrary, yet fixed, and consider the functions
$$I^k \ddf |\pk_1|^2+|\nn\vk|^2+|\nn\ve|^2+ |\Sk|^2+|\bs|^2+1,$$
where $\bs \in L^2(Q)$ be such that
\begin{align}2\ne \frac{\left(|\D \ve|-\tau_1(e) \right)^+}{|\D \ve|}\D \ve = \frac{\left(|\bs|-\tau_2(e) \right)^+}{|\bs|}\bs. \label{synarchy}
\end{align}
 Note that thus defined, $\bs$ is not uniquely determined in $\{\D \ve =0 \} \cap \{ \ted > 0 \}$. Therefore we simply set $\bs \ddf 0$ in $\{\D \ve =0 \}$. Next, we define sets $$Q_i^k \ddf \{ N^i < |\vk-\ve| \le N^{i+1} \}, \quad i= 1,\ldots,N.$$
Since $Q_i^k$ are mutually disjoint, for all $k$ we have $$\sum_{i=1}^N \int_{Q_i^k} I^k \le \sup_k \iq I^k \le C,$$ so that for every $k$ there is $i_k \in \{1,\ldots,N \}$ such that 
\begin{align}\label{tothe}
\int_{Q_{i_k}^k} I^k  \le \frac{C}{N}.
\end{align}
Finally, set $$\lambda^k \ddf N^{i_k}.$$

\paragraph{$\Li$-truncation itself.}
Due to the results~\eqref{redspa}--\eqref{redspc}, we can test in the weak formulation~\eqref{prvni} with the truncated velocity difference
\begin{align*}\wk \ddf \begin{cases}
      \vk-\ve \quad &\text{if $|\vk-\ve| \le \lambda^k$},\\
      \lambda^k \dfrac{\vk-\ve}{|\vk-\ve|} \quad &\text{if $|\vk-\ve| > \lambda^k$}.                                                           \end{cases}
\end{align*}
From~\eqref{kc} and~\eqref{kd}, it follows that
\begin{align} 
\nonumber
 \wk &\rightarrow \0 \quad \text{strongly in $L^{\infty)}(Q)$ and $L^{\infty)}(\Gamma)$,}\\
\intertext{which in conjunction with $|\nn\wk|\le 2 |\nn (\vk-\ve)|$ and convergence~\eqref{ka} lets us also assume} \label{mindb}
\wk &\rightarrow \0 \quad \text{weakly in $L^2(0,T;\Wn(\om))$.}
\end{align}
Finally, referring to~\cite[Sec.\ 2.3]{BMR07}, we only mention the last crucial property of $\wk$, namely
\begin{align*}
 \liminf_{k\to\infty} \itt \!\! \langle \dt \vk, \wk \rangle \, dt \geq 0.
\end{align*}
 Using $\wk$ as a test function in~\eqref{prvni} hence yields
\begin{align*}
 &\limsup_{k \to \infty}\iq \Bigl(\Sk \cdot \D\wk - \pk_1 \dir \wk \Bigr) \\
  & \quad = \limsup_{k \to \infty} \left[ -\itt \!\! \langle \dt \vk, \wk \rangle - \iq \Bigl( \dir \vok \cdot \wk + \nn \pk_2 \cdot \wk \Bigr)\!-\! \ig \sk \cdot \wk  \right] 
  \le 0.
\end{align*}
Due to 
\begin{align*}
|\dir \wk| \le \begin{cases}
               0 \quad &\text{in $\{ |\vk-\ve| \le \lambda^k \}$}, \\
\dfrac{2\lambda^k(|\nn \vk|+|\nn \ve|)}{|\vk-\ve|} \quad & \text{in $\{ |\vk-\ve| > \lambda^k \}$}
               \end{cases}
\end{align*}
and~\eqref{mindb}, we observe
\begin{align*}
 \limsup_{k \to \infty}& \iq (\Sk -\bs) \cdot \D\wk \le \limsup_{k\to \infty} \int_{\{ |\vk-\ve| > \lambda^k \}} \frac{2 \lambda^k}{|\vk-\ve|}|\pk_1|(|\nn \vk|+|\nn \ve|).
\end{align*}
The last estimate and Young's inequality imply
\begin{align}
\left.
\begin{aligned} \label{plyophony}
 \limsup_{k \to \infty}& \int_{\{ |\vk-\ve| \le \lambda^k \}} (\Sk -\bs) \cdot (\D\vk-\D\ve) \\ 
 &\qquad = \limsup_{k \to \infty} \Big( \iq (\Sk -\bs) \cdot \D\wk -  \int_{\{ |\vk-\ve| > \lambda^k \}} (\Sk -\bs) \cdot \D\wk \Big)  \\
 &\qquad \le C \cdot \limsup_{k\to \infty} \int_{\{ |\vk-\ve| > \lambda^k \}} \frac{\lambda^k}{|\vk-\ve|}I^k \\ 
 &\qquad = C \cdot \limsup_{k\to \infty} \Big( \int_{\{ N^{i_k+1} \geq |\vk-\ve| > N^{i_k} \}} \frac{N^{i_k}}{|\vk-\ve|}I^k + \int_{\{ |\vk-\ve| > N^{i_k+1} \}} \frac{N^{i_k}}{|\vk-\ve|}I^k \Big) \\ 
 &\qquad \le C\cdot \limsup_{k\to \infty}\Big( \int_{Q_{i_k}^k} I^k + \frac{1}{N^{i_k}} \iq I^k \Big) \\
 &\qquad \le \frac{C}{N},
\end{aligned}
\right.
\end{align}
the last step being asured by~\eqref{tothe}, with $C>0$ independent of $k$ or $N$.

\paragraph{Checking the $\llimsup$ condition.} First we show $Z^k\ddf(\Sk -\bs) \cdot (\D\vk-\D\ve)$ becomes \emph{almost} non-negative for $k\to\infty$, i.e.
\begin{align} \label{poznichkturkovi}
(Z^k)^- \rightarrow 0 \quad \text{strongly in $L^1(Q)$,}
\end{align}
where we recall $(Z^k)^- \ddf \min \{Z^k,0 \}$. Recalling furthermore our condensed notation $\Sk = \Sk(\D\vk,e^k)$ and monotonicity of $\Sk(\cdot,e^k)$ due to Lemma~\ref{monok}, we have a.e.\ in $Q$
\begin{align*}
Z^k&= (\Sk -\Sk(\D\ve,e^k)) \cdot (\D\vk-\D\ve)+(\Sk(\D\ve,e^k) -\bs) \cdot (\D\vk-\D\ve) \\[5pt]
&\geq (\Sk(\D\ve,e^k) -\bs) \cdot (\D\vk-\D\ve). 
\end{align*}
We will show
\begin{align} \label{poznichkturkovi2}
(\Sk(\D\ve,e^k) -\bs) \cdot (\D\vk-\D\ve) \rightarrow 0 \quad \text{strongly in $L^1(Q)$,}
\end{align}
which implies~\eqref{poznichkturkovi}. Let us write $Q = \{ |\D\ve| = 0\} \cup \{ |\D\ve| > 0\}$. This decomposition makes it easy to see that thanks to definitions of $\Sk$ and $\bs$, i.e.~\eqref{grafprvni} and~\eqref{synarchy}, respectively, and the fact that $(e^k)_k$ converges pointwise by~\eqref{kl}, we have  
\begin{align*}
\Sk(\D\ve,e^k) \rightarrow \bs \quad \text{a.e.\ in $Q$.}
\end{align*}
Next, again owing to~\eqref{grafprvni}, \eqref{synarchy} and Assumption~\ref{assonkoefs} we also observe
\begin{align*}
\big| \Sk(\D\ve,e^k) -\bs \big| \leq C |\D\ve|,
\end{align*}
a.e.\ in $Q$ independently of $k$. The dominated convergence theorem (mind that sequence $(\D\vk-\D\ve)_k$ is bounded in $L^2(Q)$ from~\eqref{ka}) hence yields~\eqref{poznichkturkovi2}.

Next we combine results~\eqref{plyophony} and~\eqref{poznichkturkovi}, deducing
\begin{align} \label{travis}
\limsup_{k \to \infty}& \int_{\{ |\vk-\ve| \le \lambda^k \}} \big| Z^k \big| \le \frac{C}{N}.
\end{align}
H\"older's and Chebyshev's inequalities therefore entail 
\begin{align*}
\iq \sqrt{\big| Z^k \big|} &\le \int_{\{ |\vk-\ve| \le \lambda^k \}} \sqrt{\big| Z^k \big|} + \int_{\{ |\vk-\ve| > \lambda^k \}} \sqrt{\big| Z^k \big|} \\
&\le \sqrt{|Q|} \sqrt{\int_{\{ |\vk-\ve| \le \lambda^k \}} \big| Z^k \big|} +C \sqrt{\big|\{ |\vk-\ve| > N \}\big|}\\
&\le C \Biggl(\sqrt{ \int_{\{ |\vk-\ve| \le \lambda^k \}} \big| Z^k \big|} + \frac{1}{N} \Biggr),
\end{align*}
producing, by means of~\eqref{travis},
$$\limsup_{k \to \infty} \iq \sqrt{\big| Z^k \big|} \le \frac{C}{\sqrt{N}}, $$
for every $N \in \N$, allowing us to assume 
\begin{align*}
Z^k &\to 0 \quad \text{a.e.\ in $Q$.}
\intertext{\hspace{\parindent} Now, let $\d > 0$ be arbitrary. Egoroff's theorem implies} 
Z^k &\to 0 \quad \text{strongly in $L^1(E)$}
\end{align*}
for some measurable $E\subset Q$ with $|Q\setminus E| < \d$. It means, recalling convergences~\eqref{ka} and~\eqref{kk} in the process, 
\begin{align} 0 = \lim_{k \to \infty} \int_{E} Z^k = \lim_{k \to \infty} \int_{E} (\Sk-\bs) \cdot (\D\vk - \D\ve)= \lim_{k \to \infty} \int_{E} \Sk \cdot \D\vk - \int_{E} \S \cdot \D\ve. \nonumber
\end{align}
Hence it is enough to take a sequence $\d_j \to 0$, corresponding $E^j \subset Q$ and we obtain in particular~\eqref{budesehodit}.

\paragraph{Internal energy inequality.}
In order to show~\eqref{ttreti} by passing to limit in~\eqref{druha}, it suffices to prove, due to convergences~\eqref{kc} and~\eqref{kg}--\eqref{ki},
\begin{align} \label{deity}
 \liminf_{k\to\infty} \iq \big[(\Sk \cdot\D\vk)u \big] \geq \iq \big[(\S \cdot\D\ve)u\big],
\end{align}
for  every non-negative $u\in \C([0,T];W^{1,\infty})$. But this is easy, since $(\Sk \cdot\D\vk)u \geq 0$ and $(\S \cdot\D\ve)u$ is summable in $Q$: Let $\e > 0$ be arbitrary. Then there is $\d>0$ such that for any measurable $E\subset Q$, $|Q\setminus E| < \d$, we have 
$$\int_{Q\setminus E} \big[(\S \cdot\D\ve)u\big] \leq \e.$$ 
Take the corresponding $E$ in Lemma~\ref{superconverlemma}. Then  
\begin{align*}
 \liminf_{k\to\infty} \iq \big[ (\Sk \cdot\D\vk)u\big]  \geq \liminf_{k\to\infty} \int_E \big[(\Sk \cdot\D\vk)u\big] =  \int_E \big[ (\S \cdot\D\ve)u\big] \geq \iq \big[ (\S \cdot\D\ve)u\big] - \e,
\end{align*}
implying~\eqref{deity}.

\paragraph{Initial condition for the velocity.}
Attainment of the initial condition~\eqref{initc} for the velocity is shown in a standard way; see e.g.~\cite[Sec.\ 5.4]{ja4}. We will reproduce it here mutatis mutandis just for the sake of reader's convenience.

Let $\zeta \in C^1_c([0,T))$ such that $\zeta(0)=-1$. Multiply equation~\eqref{pprvni} with $\zeta$ and integrate it over $(0,T)$. On account of the regularity results implied by~\eqref{ka}, \eqref{kb} and~\eqref{ke}, we have $\ve\in\C_w([0,T];L^2(\om))$ and for all $\w \in W^{1,\frac52}_{\n}(\om)$ we have
\begin{align} \label{verdarb}
  \int_{\om} \ve(0) \cdot \w \, dx & = \iq \big(\ve \cdot \w \zeta' - ( \S - p \I - \vov )\cdot \nn \w\zeta \big) \, dx\, dt - \ig \s \cdot \w\zeta \, dS \, dt. 
\end{align}
Doing the same thing in the approximated equation~\eqref{prvni} yield 
\begin{align*} 
  \int_{\om} \ve_0 \cdot \w \, dx & = \iq \big( \vk \cdot \w \zeta' - \big( \Sk - \pk \I - \vok \big)\cdot \nn \w\zeta \big) \, dx \, dt - \ig \sk \cdot \w\zeta \, dS \, dt. 
\end{align*}
Taking $k\to\infty$, using the properties~\eqref{ka}--\eqref{kj} and comparing the result with~\eqref{verdarb} yields 
\begin{align} \label{karfiol}
 \int_{\om} \big (\ve_0 - \ve(0)\big) \cdot \w \, dx = 0,
\end{align} 
so that $\ve(0)=\ve_0$ due to arbitrariness of $\w$.
 
Next we will show that 
\begin{align} \label{konecna}
 \vk(t) \to \ve(t) \quad \text{weakly in $L^2(\om)$ for all $t \in (0,T)$.}
\end{align}
Let $t \in (0,T)$, then sequence $(\vk(t) )_k$ is bounded in $L^2(\om)$ and we may assume that for a subsequence
$$\ve^{k_n}(t) \to \overline{\ve} \quad \text{weakly in $L^2(\om)$.} $$
 Consider~\eqref{prvni} for an arbitrary $\w \in W^{1,\frac52}_{\n}(\om)$ and multiply it with $\chi_{(0,t)}$, the characteristic function of~$(0,t)$. Then 
\begin{multline*}
 \int_{\om} \ve^{k_n}(t) \cdot\w \, dx - \int_{\om}  \ve_0 \cdot \w \, dx \\
 = \iqt \big( p^{k_n} \I + \vokn - \S^{k_n}\big) \cdot \nn \w \, dx \, dt - \igt \s^{k_n}\cdot \w \, dS \, dt,
\end{multline*}
which tends for $n \to \infty$ to 
\begin{align*}
 \int_{\om} \overline{\ve} \cdot\w \, dx - \int_{\om}  \ve_0 \cdot \w \, dx &= \iqt \big( p \I + \vov - \S\big) \cdot \nn \w \, dx \, dt - \igt \s\cdot \w \, dS \, dt \\[3pt]
&= \int_{\om} \ve(t) \cdot\w \, dx - \int_{\om}  \ve_0 \cdot \w \, dx,
\end{align*}
by the weak formulation~\eqref{pprvni}. Therefore $\overline{\ve} = \ve(t)$ and~\eqref{konecna} holds true in the end.

Regarding the strong convergence to the initial value in $L^2(\om)$, in the weak formulation~\eqref{prvni} we can take $\w = \vk$ and multiply the equation by $\chi_{(0,t)}$ for any $t \in (0,T)$, obtaining
\begin{align*}
 \| \vk(t) \|_2^2 - \| \ve_0 \|_2^2 = - \iqt \Sk \cdot \D \vk \, dx \, dt - \igt \sk \cdot \vk_{\tau} \, dS \, dt \le 0,
\end{align*}
by virtue of coercivity of $\A^k$ and $\B^k$ from Lemma~\ref{monok}.

Recalling $\ve\in\C_w([0,T];L^2(\om))$, $\ve(0)=\ve_0$ and then adding~\eqref{konecna} with the lower semicontinuity of the norm finally yields
\begin{align} \label{alneri}
 \lim_{t \to 0_+ } \left\| \ve(t) - \ve_0 \right\|_2^2  = \lim_{t \to 0_+ } \left\| \ve(t) \right\|_2^2 - \left\| \ve_0 \right\|_2^2  \le  \lim_{t \to 0_+ } \liminf_{k \to \infty} \| \vk(t) \|_2^2 - \left\| \ve_0 \right\|_2^2 \le  0.
\end{align}

\paragraph{Initial condition for the internal energy.} \label{intene}

First we note that~\eqref{primea} implies $$E \in \C([0,T];W^{-1,\mfrac{10}{9}}(\om)),$$ and using the same procedure as towards~\eqref{karfiol}, we could deduce 
\begin{align} \nonumber
E(0)=\dfrac{|\ve_0|^2}{2} +e_0.
\end{align} 
The already proven~\eqref{alneri} therefore yields
\begin{align} \label{vifamath}
\lim_{t\to0_+} \ii e(t)\, dx = \ii e_0 \, dx.
\end{align}
Now we finally recall the inequality~\eqref{treti}. From the convergence results~\eqref{kb}--\eqref{kk} it follows that for a.e.\ $t\in(0,T)$ and any non-negative $u\in W^{1,\infty}(\om)$, we have
\begin{align*}
\ii \sqrt{e(t)} \, u \, dx - \ii \sqrt{e_0} \,u \, dx   + \iqt \Bigl( -\sqrt{e}\, \ve + \dfrac{\ke \nn e}{2 \sqrt{e}} \Bigr)\cdot \nn u \, dx \, dt \geq 0.
\end{align*}
This inequality patently implies 
\begin{align} \nonumber
\essliminf_{t\to0_+} \ii \sqrt{e(t)} \, u \, dx \geq \ii \sqrt{e_0} \,u \, dx,
\end{align}
for any non-negative $u\in W^{1,\infty}(\om)$, whence also for any non-negative $u\in L^2(\om)$ by the density argument, so that
\begin{align} \label{vifamath2}
\essliminf_{t\to 0_+} \ii \sqrt{e(t)} \sqrt{e_0} \, dx \geq \ii e_0 \, dx.
\end{align}
Since $e(t) - e_0=\big(\sqrt{e(t)} + \sqrt{e_0}\big)\big(\sqrt{e(t)} - \sqrt{e_0}\big)$, combining ultimately~\eqref{vifamath} with~\eqref{vifamath2} produces
\begin{align*}
\| e(t) - e_0 \|_1^2 \le \| \sqrt{e(t)} + \sqrt{e_0} \|_2^2  \| \sqrt{e(t)} - \sqrt{e_0} \|_2^2 \le C \ii \big( e(t) + e_0 - 2 \sqrt{e(t)}  \sqrt{e_0} \big) \, dx \to 0
\end{align*}
in the essential sense for $t\to0_+$.
\end{proof}

\section{Appendix}
This section is dedicated for the postponed proof of Lemma~\ref{converlemma}.

\begin{customthm}{\ref{converlemma}} \label{prvnilemma}
Let $\A$ be of the form~\eqref{satis1} and $U \subset Q$ be measurable. Consider sequences $(\S^k)_k$, $(\D^k)_k$ and $(e^k)_k$ of measurable functions on $U$, satisfying
 \begin{align}
  &&(\S^k,\D^k,e^k) &\in \A &&\text{a.e.\ in $U$ for all $k\in\N$}, \label{shiftX} \\
  && \S^k &\rightarrow \S && \text{weakly in $L^2(U)$}, \label{weakSX}\\
  && \D^k &\rightarrow \D && \text{weakly in $L^{2}(U)$},\label{weakDX} \\
  && e^k &\rightarrow e && \text{a.e.\ in $U$}, \label{shift2X} \\
  &&\limsup_{k \to \infty} \int_U \S^k \cdot \D^k &\le \int_U \S \cdot \D.&& \label{chivaldoreX}
 \end{align}
Then $(\S,\D,e) \in \A$ a.e.\ in $U$ and $\S^k \cdot \D^k \to \S \cdot \D$ weakly in $L^1(U)$.
\end{customthm}

\begin{proof}
Let $\Aa, \BB \in L^2(U)$ be arbitrary. Denote
\begin{align*}
U_1&\ddf \{\tau_1(e)=0\}, \\
U_2&\ddf \{\tau_1(e)>0\}
\end{align*}
and recall 
\begin{align*} 
 (\Aa,\BB,e)\in\A \quad \Longleftrightarrow \quad \begin{cases}
    \BB=\BB(\Aa)= \dfrac{1}{2\ne} \cdot \dfrac{\left(|\Aa|-\tau_2(e) \right)^+}{|\Aa|}\Aa \quad \text{\!in $U_1$}, \\
\Aa= \Aa(\BB)= 2\ne\cdot \dfrac{\left(|\BB|-\tau_1(e) \right)^+}{|\BB|}\BB \quad \text{in $U_2$}.                                  
\end{cases}
\end{align*}
Our first aim is therefore to show 
\begin{align}
\begin{aligned}
\D &= \frac{1}{2\nu(e)}\cdot \frac{(|\S| - \ted)^+}{|\S|}\S  \quad \text{in } U_1,\\
\S &= 2\nu(e) \cdot \frac{(|\D| - \tej)^+}{|\D|}\D \quad \text{\!in } U_2.
\end{aligned}
\label{tojeono}
\end{align}
The technical hindrance in our situation is that $U_1 \neq \{\tau_1(e^k)=0\}$, i.e.\ $U_2 \neq \{\tau_1(e^k)>0\}$ in general. However, due to compactness of $(e^k)_k$ and the tractable behavior of the function $\nu(\cdot)$ (see Assumption~\ref{assonkoefs}), we are able to solve the problem using practically Minty's method based on monotonicity.

 Since both $\BB(\Aa)$ and $\Aa(\BB)$ are monotone and $\nu(\cdot) \geq c_1 >0$ by assumption~\eqref{visco}, we have
\begin{align}
I_1^k&\ddf(\S^k - \Aa) \cdot \left( \frac{1}{2\nu(e^k)}\cdot \frac{(|\S^k| - \ted)^+}{|\S^k|}\S^k  - \frac{1}{2\nu(e^k)} \cdot\frac{(|\Aa| - \ted)^+}{|\Aa|}\Aa  \right) \geq 0\quad \text{\,in } U_1, \label{I1} \\
I_2^k&\ddf(\D^k - \BB) \cdot \left( 2\nu(e^k)\cdot \frac{(|\D^k| - \tej)^+}{|\D^k|}\D^k  - 2\nu(e^k)\cdot \frac{(|\BB| - \tej)^+}{|\BB|} \BB \right) \geq 0\quad \text{\!in } U_2. \label{I2}
\end{align}
Condition~\eqref{shiftX} allows us to expand $I_1^k$ and $I_2^k$ as
\begin{align*}
I_1^k&=(\S^k - \Aa) \cdot \left( \frac{1}{2\nu(e^k)}\cdot \frac{(|\S^k| - \ted)^+}{|\S^k|}\S^k  - \frac{1}{2\nu(e^k)}\cdot \frac{(|\S^k| - \tau_2(e^k))^+}{|\S^k|}\S^k\right) \\ &+(\S^k - \Aa) \cdot \left( \frac{(|\D^k| - \tau_1(e^k))^+}{|\D^k|}\D^k  - \D^k\right) + (\S^k - \Aa) \cdot\left( \D^k
 -\frac{1}{2\nu(e^k)} \cdot\frac{(|\Aa| - \ted)^+}{|\Aa|}\Aa  \right)\\
&\dfe J_1^k + J_2^k + J_3^k
\intertext{and}
I_2^k&=(\D^k - \BB) \cdot \left( 2\nu(e^k)\cdot \frac{(|\D^k| - \tej)^+}{|\D^k|}\D^k - 2\nu(e^k)\cdot \frac{(|\D^k| - \tau_1(e^k))^+}{|\D^k|}\D^k \right)  \\
& + (\D^k - \BB) \cdot \left( \frac{(|\S^k| - \tau_2(e^k))^+}{|\S^k|}\S^k - \S^k \right)+ (\D^k - \BB) \cdot \left(
\S^k - 2\nu(e^k)\cdot \frac{(|\BB| - \tej)^+}{|\BB|} \BB \right)\\
&\dfe J_4^k + J_5^k + J_6^k.
\end{align*}
From definitions~\eqref{I1} and~\eqref{I2} it therefore follows that 
\begin{align}\label{dokopy}
 0\leq \int_{U_1} I_1^k + \int_{U_2} I_2^k = \int_{U_1} J_1^k + \int_{U_1} J_2^k + \int_{U_1} J_3^k + \int_{U_2} J_4^k + \int_{U_2} J_5^k + \int_{U_2} J_6^k.
\end{align}
Let us now show 
\begin{align}\label{J1245}
 \lim_{k\to \infty} \left( \int_{U_1} J_1^k + \int_{U_1} J_2^k + \int_{U_2} J_4^k + \int_{U_2} J_5^k \right) =0.
\end{align}
Since the function $f_a(x)\ddf(a-x)^+$ is a 1--Lipschitz mapping for any $a\in\R$, we deduce
\begin{equation*}
\left| (|\S^k| - \ted )^+ - (|\S^k| - \tau_2(e^k))^+ \right| \leq |\ted - \tau_2(e^k)|,
\end{equation*}
which, due to the continuity of $\tau_2$ and the pointwise convergence~\eqref{shift2X}, implies
\begin{align*}
(|\S^k| - \ted )^+ - (|\S^k| - \tau_2(e^k))^+\rightarrow 0 \quad \text{strongly in $L^{\infty)}(U)$},
\end{align*}
entailing 
\begin{align} \label{zimajedna}
J_1^k \rightarrow 0 \quad \text{strongly in $L^1(U_1)$},
\end{align} 
 owing to~\eqref{weakSX}, H\"older's inequality and Assumption~\ref{assonkoefs}, controlling the viscosity $\nu$. Similarly, in $U_1$ we have
\begin{equation*}
\begin{aligned}
\left| \frac{(|\D^k| - \tau_1(e^k) )^+}{|\D^k|} \D^k - \D^k \right| &= \left| \frac{(|\D^k| - \tau_1(e^k) )^+}{|\D^k|} \D^k - \frac{(|\D^k| - \tej )^+}{|\D^k|} \D^k \right| \\ &= \left| (|\D^k| - \tau_1(e^k) )^+ - (|\D^k| - \tej )^+ \right| \\& \leq |\tau_1(e^k) - \tej|,\phantom{\int}
\end{aligned}
\end{equation*}
whence 
\begin{align*}
\frac{(|\D^k| - \tau_1(e^k) )^+}{|\D^k|} \D^k - \D^k \rightarrow \0 \quad \text{strongly in $L^{\infty)}(U_1)$},
\end{align*}
implying this time
\begin{align} \label{zimadva}
J_2^k \rightarrow 0 \quad \text{strongly in $L^1(U_1)$},
\end{align}  
for the same reason as in~\eqref{zimajedna}. As for $J_4^k$ and $J_5^k$, these terms we can treat in the very analogous fashion, thus proving~\eqref{J1245}. If we plug this result back into~\eqref{dokopy}, we obtain 
\begin{align} \nonumber
0\leq \limsup_{k \to \infty} \left(\int_{U_1} J_3^k + \int_{U_2} J_6^k \right) &= \limsup_{k \to \infty} \Biggl[\int_{U_1}(\S^k - \Aa) \cdot\left( \D^k
 -\frac{1}{2\nu(e^k)} \cdot\frac{(|\Aa| - \ted)^+}{|\Aa|}\Aa  \right) \\&+\int_{U_2} (\D^k - \BB) \cdot \left(
\S^k - 2\nu(e^k)\cdot \frac{(|\BB| - \tej)^+}{|\BB|} \BB \right)\Biggr]. \label{J36}
\end{align}
We have yet to replace $\nu(e^k)$ with $\nu(e)$ on the right hand side, which is easy, given that boundedness of $\nu$ from below yields
\begin{equation*}
\left| \left( \frac{1}{\nu(e)} - \frac{1}{\nu(e^k)} \right)\cdot \frac{(|\Aa| - \ted)^+}{|\Aa|} \Aa \right|^2 \leq C |\Aa|^2
\end{equation*}
and $|\Aa|^2$ is in $L^1$ so that we may use the dominated convergence theorem (mind the continuity of $\nu$) to get
\begin{align}
\left( \frac{1}{\nu(e)} - \frac{1}{\nu(e^k)} \right)\cdot \frac{(|\Aa| - \ted)^+}{|\Aa|}\Aa&\rightarrow \0 \quad \text{strongly in $L^2(U)$}
\label{Amaj}
\intertext{and likewise}
\left( \nu(e) - \nu(e^k) \right)\cdot \frac{(|\BB| - \tej)^+}{|\BB|}\BB &\rightarrow \0 \quad \text{strongly in $L^2(U)$}.
\label{Bmaj}
\end{align}
Using~\eqref{Amaj} and~\eqref{Bmaj} in~\eqref{J36}, we can further estimate
\begin{align} \nonumber
0\leq \limsup_{k \to \infty} \Biggl[ & \int_{U_1}(\S^k - \Aa) \cdot\left( \D^k
 -\frac{1}{2\ne} \cdot\frac{(|\Aa| - \ted)^+}{|\Aa|}\Aa  \right) \\  +&\int_{U_2} (\D^k - \BB) \cdot \left(
\S^k - 2\ne\cdot \frac{(|\BB| - \tej)^+}{|\BB|} \BB \right)\Biggr]. \label{mistra}
\end{align}
Finally, recalling our assumptions~\eqref{weakSX}, \eqref{weakDX}, \eqref{chivaldoreX} and exploiting $U_2 = U\setminus U_1$, we pass to limit in~\eqref{mistra}, thence
\begin{align}
0 \leq \! \int_{U_1} \! (\S - \Aa) \cdot \left( \! \D - \frac{1}{2\nu(e)}\cdot \frac{(|\Aa| - \ted)^+}{|\Aa|}\Aa  \! \right) + \int_{U_2} \! (\D - \BB) \cdot \left( \! \S - 2\nu(e)\cdot \frac{(|\BB| - \tej)^+}{|\BB|} \BB \! \right).
\label{skoro2}
\end{align}
Now we take $\Aa \ddf \S \pm \e \CC$ for $\e>0$ and an arbitrary $\CC \in L^2(U)$, and $\BB \ddf \D$. Like in the standard Minty method, after $\e\to 0_+$ we deduce
\begin{align*}
0 = \int_{U_1}  \CC \cdot \left( \D - \frac{1}{2\nu(e)}\cdot \frac{(|\S| - \ted)^+}{|\S|}\S  \right),
\end{align*}
which implies, due to the arbitrary nature of $\CC$, 
\begin{align*}
\D &= \frac{1}{2\nu(e)}\cdot \frac{(|\S| - \ted)^+}{|\S|}\S  \quad \text{in } U_1.
\intertext{Similarly, if we in~\eqref{skoro2} set $\Aa \ddf \S$ and $\BB \ddf \D \pm \e \CC$, we infer  }
\S &= 2\nu(e)\cdot \frac{(|\D| - \tej)^+}{|\D|}\D  \quad \text{\!in } U_2,
\end{align*}
i.e.~\eqref{tojeono}, in other words $(\S, \D, e) \in \A$.

Towards the second part of the lemma, namely
\begin{align} \label{dhl}
\S^k \cdot \D^k \rightarrow \S \cdot \D \quad \text{weakly in $L^1(U)$},
\end{align} 
let us set $\Aa \ddf \S$ and $\BB \ddf \D$ in~\eqref{I1}~and \eqref{I2}, respectively. Since this choice makes the right-hand side of~\eqref{skoro2} vanish, the procedure between~\eqref{dokopy} and~\eqref{skoro2} implies
\begin{equation*}
\lim_{k\to \infty}\left(\int_{U_1} I_1^k + \int_{U_2} I_2^k \right)= 0.
\end{equation*}
Due to non-negativity of $I_1^k$ and $I_2^k$ we therefore obtain $I_1^k \to 0$ strongly in $L^1(U_1)$, and since $J_3^k = I_1^k - J_1^k - J_2^k$ and the right-hand side converges to zero strongly in $L^1(U_1)$ due to~\eqref{zimajedna} and~\eqref{zimadva}, we acquire
\begin{align*} 
J_3^k = (\S^k - \S) \cdot \left( \D^k - \frac{1}{2\nu(e^k)}\cdot \frac{ (|\S| - \ted)^+}{|\S|}\S \right) \rightarrow 0 \quad \text{strongly in $L^1(U_1)$}.
\end{align*} 
By~\eqref{Amaj}, we may replace $\nu(e^k)$ with $\nu(e)$ in the above convergence, turning the resulting term into $\D$ by~\eqref{tojeono}. But this tells us nothing less than 
\begin{align*} 
(\S^k - \S) \cdot (\D^k - \D) \rightarrow 0 \quad \text{strongly in $L^1(U_1)$}.
\end{align*} 
Completely analogously, from $J_6^k = I_2^k - J_4^k - J_5^k$, we can obtain the above convergence also in $L^1(U_2)$, i.e.\ in the end 
\begin{align*} 
(\S^k - \S) \cdot (\D^k - \D) \rightarrow 0 \quad \text{strongly in $L^1(U)$}.
\end{align*}
Now exploiting the assumed weak convergences~\eqref{weakSX} and~\eqref{weakDX} yields~\eqref{dhl}, considering
\begin{align*} 
\S^k\cdot\D^k=(\S^k - \S) \cdot (\D^k - \D) + \S \cdot (\D^k - \D)+\S^k\cdot\D \rightarrow \S\cdot\D \quad \text{weakly in $L^1(U)$}.
\end{align*}
\end{proof}

\begin{acknowledgements}
 \noindent E.\ Maringov\'a was supported by the ERC-CZ project LL1202 financed by the Ministry of Education, Youth and Sports of the Czech Republic and additionally by the grant GAUK 584217. J.\ \v Zabensk\'y was partially supported by the ERC-CZ project LL1202.
\end{acknowledgements}

%
%
%
%

\bibliographystyle{plainnat}

\end{document}